\documentclass[11pt]{amsart}
\usepackage{amsmath}
\usepackage{amssymb}
\usepackage{amsthm}
\usepackage{fancyhdr}
\usepackage{enumerate}
\usepackage{eucal}
\usepackage{fixmath}
\usepackage{tikz}
\usetikzlibrary{patterns}
\usepgflibrary{patterns}
\usepackage[square]{natbib}
\usepackage{url}

\setcitestyle{numbers}

\title[Rationally Smooth Affine Schubert Varieties]{Pattern Characterization of Rationally Smooth Affine Schubert Varieties of Type $A$}
\author{Sara Billey \and Andrew Crites}
\date{\today}
\thanks{The authors acknowledge support from grant DMS-0800978 from the National Science Foundation.}
\address{Department of Mathematics, University of Washington, Box 354350, Seattle, Washington, 98195-4350, billey@uw.edu, acrites@uw.edu}

\newtheorem{thm}{Theorem}[section]
\newtheorem*{thm*}{Theorem}
\newtheorem{lemma}[thm]{Lemma}
\newtheorem{prop}[thm]{Proposition}
\newtheorem{conj}{Conjecture}
\newtheorem{cor}[thm]{Corollary}
\newtheorem{remark}[thm]{Remark}
\numberwithin{equation}{section}

\theoremstyle{definition}
\newtheorem{definition}[thm]{Definition}

\newcommand{\poin} {P}
\newcommand{\scrR} {\mathcal{R}}
\newcommand{\scrA} {\mathcal{A}}
\newcommand{\Z} {\mathbb{Z}}
\newcommand{\C} {\mathbb{C}}
\newcommand{\N} {\mathbb{N}}
\newcommand{\Inv} {\mathrm{Inv}}
\newcommand{\Stab} {\mathrm{Stab}}
\newcommand{\pt} {\mathrm{pt}}
\newcommand{\GL} {\mathrm{GL\,}}
\newcommand{\affinv}{\Inv_{\wt{S}_n}}
\newcommand{\ol}[1] {\overline{#1}}
\newcommand{\wt}[1] {{\widetilde{#1}}}
\newcommand{\wh}[1] {\widehat{#1}}
\newcommand{\myomega}{w}
\newcommand{\mynu}{v}
\newcommand{\mymu}{u}
\newcommand{\mychi}{x}
\newcommand{\mysigma}{\sigma}

\begin{document}

\begin{abstract}
Schubert varieties in finite dimensional flag manifolds $G/P$ are a
well-studied family of projective varieties indexed by elements of the
corresponding Weyl group $W$.  In particular, there are many tests for
smoothness and rational smoothness of these varieties.  One key result
due to Lakshmibai-Sandhya is that in type $A$ the smooth Schubert
varieties are precisely those that are indexed by permutations that
avoid the patterns $4231$ and $3412$.  Recently, there has been a
flurry of research related to the infinite dimensional analogs of flag
manifolds corresponding with $G$ being a Kac-Moody group and $W$ being
an affine Weyl group or parabolic quotient.  In this paper we study
the case when $W$ is the affine Weyl group of type $A$ or the affine
permutations.  We develop the notion of pattern avoidance for affine
permutations.  Our main result is a characterization of the rationally
smooth Schubert varieties corresponding to affine permutations in
terms of the patterns $4231$ and $3412$ and the twisted spiral
permutations.

\end{abstract}

\maketitle



\section{Introduction}
\label{intro}

The study of Schubert varieties and their singular loci incorporates tools from algebraic geometry, representation theory and combinatorics.
One celebrated result in this area due to Lakshmibai-Sandhya is that in classical type $A$ the smooth Schubert 
varieties are precisely those that are indexed by permutations that avoid the patterns $4231$ and $3412$ \cite{LakSan}, see also \cite{ryan,Wolper}.
A second important theorem in this area concerns a weaker notion than smoothness based on cohomology, called rational smoothness.
In general, smoothness implies rational smoothness, but not conversely.
For classical Schubert varieties, Peterson showed smoothness is equivalent to rational smoothness precisely in types $A,D,E$.
Recently, there has been a surge of research activity related to affine Schubert varieties
\cite{BilleyMitchell,lakshmibai-kreiman,KuttlerLakshmibai,lam,LLMS:08,lam-schilling-shimo,magyar}.
It is natural to ask how properties of smoothness, rational smoothness, singular loci and tangent spaces
for affine Schubert varieties relate to their classical counterparts.

In this paper we give a criterion for detecting rationally smooth Schubert varieties in affine type $A$.
These varieties are indexed by the set of affine permutations, denoted $\wt{S}_{n}$.
Generalizing the theorem of Lakshmibai-Sandhya, we show that the patterns $4231$ and $3412$ can be interpreted as patterns for affine permutations.
A permutation avoiding these two patterns will again index a rationally smooth affine Schubert variety.
However, there is an infinite family of affine permutations in $\wt{S}_{n}$ which contain $3412$ and yet index rationally smooth affine Schubert varieties.
These varieties are related to the spiral varieties studied by Mitchell \cite{Mitchell}, see also \cite{BilleyMitchell}.
Thus, our main result is a complete characterization of the rationally smooth Schubert varieties in $G/B$ in type $\wt{A}_n$.

\begin{thm}
\label{thm:front} Let $\myomega\in\wt{S}_n$ for $n\geq 3$.  The Schubert
variety $X_{\myomega}$ is rationally smooth if and only if one of the
following hold:
\begin{enumerate}
\item $\myomega$ avoids the patterns 3412 and 4231,
\item $\myomega$ is a twisted spiral permutation (defined in Section~\ref{s:spirals}).
\end{enumerate}
\end{thm}
Note, $\wt{S}_2$ is the infinite dihedral group.  It follows that
$X_{w}$ is rationally smooth for all $w \in \wt{S}_2$.  Hence,
throughout this paper we will assume $n\geq 3$ unless otherwise
specified.  

Every point in $X_{\myomega}$ which is not rationally smooth must be singular, hence we get a necessary condition to detect singular affine Schubert varieties.

\begin{cor}\label{cor:smoothness}
Let $\myomega\in\wt{S}_n$ for $n\geq 3$.
If $\myomega$ contains either a 3412 or 4231 pattern, then $X_{w}$ is singular.
\end{cor}

It was shown in \cite{Crites1} that there are only a finite number of 3412 avoiding affine permutations in $\wt{S}_{n}$.
Thus, there exists a finite number of smooth Schubert varieties indexed by $\wt{S}_{n}$.
Using this fact, we have verified the following conjecture up to $n=5$.

\begin{conj}\label{conj:smooth.char}
Let $\myomega\in\wt{S}_n$.  The Schubert variety $X_{w}$ is smooth if and only if $\myomega$ avoids 3412 and 4231.
\end{conj}

This paper is organized as follows.
In Section~\ref{s:background}, we present the basic definitions and constructions needed in this paper.
In Section~\ref{One Direction}, we begin the proof of Theorem~\ref{thm:front} by first discussing the case that $\myomega$ avoids the two patterns.
In particular, we show that the Poincar\'e polynomial of the cohomology ring of $X_{\myomega}$ factors into
palindromic factors as Gasharov showed in the classical type $A$ case \cite{Gasharov98,GR2000}.
By a theorem of Carrel-Peterson~\cite{CarrellPeterson}, this implies $X_{\myomega }$ is rationally smooth.
In Section~\ref{Converse}, we prove that if $\myomega$ contains the pattern 4231,
then the Poincar\'e polynomial fails to be palindromic at rank 1.
In Section~\ref{s:pictures} we develop the theory of affine Bruhat pictures, which will be used in Section~\ref{s:3412}
to show that if $\myomega$ contains the pattern 3412  then it is either a twisted spiral permutation or 
 we can identify a point in $X_{\myomega}$ which is not rationally smooth.
This will complete the proof of Theorem~\ref{thm:front}.  We conjecture the rationally singular points obtained in the proof are maximal singular points.   
Finally, in Section~\ref{conjectures} we discuss some directions for further research.

\section{Background and Notation}
\label{s:background}

In this section, we will begin by collecting common notation and background necessary to prove the main theorem.
We will proceed from general facts about Schubert varieties and Coxeter groups to specifics of the particular Coxeter group $\wt{S}_{n}$.

\subsection{Schubert Varieties}\label{s:schubertvarieties}
Let $G$ be a semisimple Lie group such as $GL_{n}({\mathbb{C}})$.
Let $B$ be a Borel subgroup of $G$.
For $GL_{n}$, we can take $B$ to be the upper triangular matrices.
The cosets $G/B$ form the points of a \emph{flag variety}.
For each $G$, there is an associated finite Weyl group $W$.
For each $\myomega \in W$, we obtain the \emph{Schubert variety} $X_{\myomega}$ by taking the Zariski closure of the orbit $B\cdot e_{\myomega}$,
where $e_{\myomega}$ represents $\myomega$ embedded in $G$.
There are many good references on Schubert varieties including \cite{Brion:flag,Fulton:97,fulton-harris,GL-book,harris,Hum-LAG,wiki:flags}.

For each semisimple Lie group $G$, there is an associated Kac-Moody group $\wt{G}$ and an associated affine Weyl group $\wt{W}$
obtained from $W$ by adding one additional generator $s_{0}$.
Weyl groups and affine Weyl groups are special cases of Coxeter groups.
For $G=GL_{n}$, the associated $\wt{G}$ is the set of invertible matrices with entries in $\mathbb{C}((z))$, the field of rational functions of $z$.
The affine analog of the flag variety is $\wt{G}/\wt{B}$, where $\wt{B}$ is the set of all invertible matrices, $g$, with entries in $\mathbb{C}[[z]]$,
the formal power series in $z$, with the property that $\left.g\right|_{z=0}$ is upper triangular.
The affine Schubert varieties are the Zariski closures of the orbits $\wt{B} \cdot e_{\myomega}$ for $\myomega$ in $\wt{W}$.
We use the same notation $X_{\myomega}$ for Schubert varieties and affine Schubert varieties,
but the index comes from an affine Weyl group if $X_{\myomega}$ is an affine Schubert variety.
For background on Kac-Moody groups and affine Schubert varieties we recommend \cite{Kumar}.

\subsection{Coxeter Groups}
\label{sub:coxeter}
For a general reference on Coxeter groups, see \cite{BB:05} or \cite{HumphreysCoxGrp}.
Let $W$ be a Coxeter group generated by a finite set $S$ with relations of the form $(s_{i}s_{j})^{m_{ij}}$,
where each $m_{ii}=1$ and $m_{ij}\geq 2$ otherwise.
Any $\myomega\in W$ can be written as a product of elements from $S$ in infinitely many ways.
Every such product will be called an \emph{expression} for $\myomega$.
Any expression of minimal length will be called a \emph{reduced expression},
and the number of letters in such an expression will be denoted $\ell(\myomega)$, the \emph{length} of $\myomega$.
Call any element of $S$ a \emph{simple reflection} and any element conjugate to a simple reflection, a \emph{reflection}.
There is a partial order on any Coxeter group $W$ defined as the transitive closure of the covering relations,
where $u\lessdot\mynu$ if there exists a reflection $t$ with $ut=\mynu$ and $\ell(u)=\ell(\mynu)-1$.
This partial order is called the \emph{Bruhat order} on $W$.
Under Bruhat order, $W$ is a ranked poset, ranked by the length function.

Given two elements $\mymu\le\myomega$ in $W$ related by Bruhat order,
let $[\mymu,\myomega]=\{\mynu\in W:\mymu\le\mynu\le\myomega\}$ denote the interval between $\mymu$ and $\myomega$.
Call the interval $[0,\myomega]$ the \emph{order ideal} of $\myomega$, where $\mymu=0$ is the identity element of $W$.  
The \emph{Poincar\'e polynomial} $P_{\myomega}(q)$ is the rank generating function for the order ideal of $\myomega$.
Specifically, $$\poin_\myomega(q)=\sum_{\mynu\le \myomega}q^{\ell(\mynu)}.$$
This polynomial gets its name from the fact that  $\poin_\myomega(q^2)$ is the Poincar\'e polynomial of the cohomology ring of $X_\myomega$.
Note that  $\mynu\le\myomega$ if and only if $\mynu^{-1}\le\myomega^{-1}$, so we have $\poin_\myomega(q)=\poin_{\myomega^{-1}}(q)$.

We use the following theorem due to Carrell-Peterson to define \emph{rational smoothness}
in terms of the combinatorics of the Bruhat order and Coxeter groups.
Let $T$ be the set of all reflections in $W$, and let $\scrR(\mychi,\myomega)=\{t\in T:\mychi<\mychi t\le\myomega\}$.
For any polynomial $f(q)$ of degree $n$, call $f$ \emph{palindromic} if $f(q)=t^nf(q^{-1})$.

\begin{prop}{\textup{\cite[Theorem E]{CarrellPeterson}}}
\label{CarrellPeterson}
Let $W$ be an (affine) Weyl group.
Let $\myomega\in W$.
Then the following are equivalent.
\begin{enumerate}
\item The (affine) Schubert variety $X_{\myomega }$ is rationally smooth.
\item $\poin_\myomega(q)$ is palindromic.
\item $\#\scrR(\mychi,\myomega)=\ell(\myomega)-\ell(\mychi)$ for all $\mychi\le\myomega$.
\end{enumerate}
\end{prop}

The main goal of this paper is to determine for which $\myomega\in\wt{S}_n$, $\poin_\myomega(q)$ is palindromic.
Each $\myomega$ with a palindromic Poincar\'e polynomial will factor using a parabolic decomposition as follows.
Given a subset $J\subseteq S$ of the generators for a Coxeter group $W$,
we can define $W_J$ to be the subgroup of $W$ generated by the elements of $J$.
$W_J$ is called the \emph{parabolic subgroup generated by $J$}.
Each coset in $W/W_J$ contains a unique element of minimal length \cite[Proposition 1.10]{HumphreysCoxGrp}.
The set of all minimal length coset representatives is denoted $$W^J=\{\myomega\in W|\ell(\myomega s)>\ell(\myomega)\text{ for all }s\in J\}.$$
Thus, we often identify $W^J$ with these cosets.
We will also need to use left cosets of $W_J$.
Let ${}^JW$ denote the minimal length left coset representatives in $W_J\backslash W$.
Hence, $${}^{J}W=\{\myomega\in W|\ell(s\myomega)>\ell(\myomega)\text{ for all }s\in J\}.$$
Bruhat order on $W$ induces partial orders on ${}^JW$ and $W^J$.
When we wish to refer to the Poincar\'{e} polynomial of a minimal length coset representative
$\myomega$ for the induced order on either of the quotients ${}^JW$ or $W^J$,
we will denote it by ${}^J\poin_\myomega$ or $\poin^J_\myomega$, respectively.
The \emph{parabolic decomposition} for elements of $W$ is given as follows.

\begin{prop}{\textup{\cite[Proposition 2.4.4,2.5.1]{BB:05}}}
\label{decomp}
For every $\myomega\in W$ there exists a unique $u\in W_J$ and a unique $v\in{}^JW$
such that $\myomega=u \cdot v$ and $\ell(\myomega)=\ell(u)+\ell(v)$.
Moreover, the map $\myomega\mapsto v$ is order preserving
as a map from $W$ to the set of minimal length coset representatives.
\end{prop}

\begin{lemma}{\textup{\cite[Lemma 2.4.3]{BB:05}}}
\label{parabolicredexp}
An element $\myomega\in W$ will be a minimal length coset representative for $W_J\myomega$ if and only if
no reduced expression for $\myomega$ begins with an element of $J$.
In the case where $J=S\backslash\{s_i\}$ and $\myomega$ is not the identity,
this implies every reduced expression for $\myomega$ must begin with $s_i$.
\end{lemma}

\begin{definition}
\label{findingmwJ}   
Let $\myomega=s_{i_1}\cdots s_{i_p}$ be any reduced expression for $\myomega \in W$.
Let $m(\myomega,J)$ denote the element obtained by the following greedy algorithm.
Initially set $m$ to be the identity element.
Then, reading left to right, inductively set $m$ to be $m s_{i_{q}}$ whenever $s_{i_q}$ is in $J$ and $\ell(m s_{i_q})>\ell(m)$.
Finally, set $m(\myomega,J)=m$.
\end{definition}
Computing $m(\myomega,J)$ is equivalent to taking the Demazure product from \cite[Definition 3.1]{KnutsonMiller04} of the subword of
$s_{i_1}\cdots s_{i_p}$ consisting of all of the letters from $J$.

\begin{prop}\label{prop:}{\textup{\cite[Theorem 2.2]{ParabolicMap}}}
For every $\myomega\in W$ and parabolic subgroup $W_{J} \subset W$, the element $m(\myomega,J)$ is the unique maximal element in $[0,\myomega]\cap W_J$.
In particular, $m(\myomega,J)$ is well-defined, independent of the choice of reduced expression for $\myomega$.   
\end{prop}

\begin{prop}{\textup{\cite[Theorem 6.4]{BilleyPostnikov}}}
\label{BilleyPostDecomp}
Suppose $\myomega\in W$ has the parabolic decomposition $\myomega=\mymu\cdot \mynu$ with $\mymu\in W_J$ and $\mynu\in{}^JW$.
If $\mymu=m(\myomega,J)$, then $\poin_\myomega(q)=\poin_\mymu(q)\cdot{}^J\poin_\mynu(q)$.
\end{prop}

\subsection{Affine Permutations}\label{intro:affineperms}
Let $S_n$ denote  the set of all bijections taking $[n]=\{1,2,\dotsc,n \}$ to itself.
We denote $w \in S_{n}$ by its one-line notation $w=[w_{1},w_{2},\dotsc , w_{n}]$ if $w$ maps $i$ to $w_{i}$.
As a Coxeter group, $S_{n}$ is generated by the adjacent transpositions $s_{i}$ interchanging $i$ with $i+1$ with relations
$s_{i}s_{j}=s_{j}s_{i}$ for $|i-j|\ge2$ and $s_{i}s_{i+1}s_{i}=s_{i+1}s_{i}s_{i+1}$ for all $1\leq i\leq n-2$.
Note, $S_{n}$ is the Weyl group of type $A$ that corresponds with the Lie group $G=GL_{n}(\mathbb{C})$.

Let $\wt{S}_n$ denote of the set of all bijections $\myomega:\Z\to\Z$ with $\myomega(i+n)=\myomega(i)+n$ for all $i\in\Z$
and $\sum_{i=1}^n\myomega(i)=\binom{n+1}{2}$.
$\wt{S}_n$ is called the \emph{affine symmetric group}, and the elements of $\wt{S}_n$ are called \emph{affine permutations}.
This definition of affine permutations  appeared in \cite[\S3.6]{Lusztig83} and was then developed in \cite{Shi}.

We can view an affine permutation in its one-line notation as the infinite string
$$\cdots\myomega_{-1}\myomega_0\myomega_1\myomega_2\cdots\myomega_n\myomega_{n+1}\cdots,$$ where $\myomega_i=\myomega(i)$.
An affine permutation is completely determined by its action on any window of $n$ consecutive indices $[\myomega_i,\myomega_{i+1},\dotsc,\myomega_{i+n-1}]$.
In particular, it is enough to record the base window $[\myomega_1,\dotsc,\myomega_n]$ to capture all the information about $\myomega$.
Sometimes however, it will be useful to write down a larger section of the one-line notation or use a different window.

We can also view an affine permutation as a matrix.
As mentioned in Section~\ref{s:schubertvarieties}, the entries of the matrix will come from $\C((z))$.
For any $\myomega=[\myomega_1,\dots,\myomega_n]\in\wt{S}_n$,
write $\myomega_i=a_i+nb_i$, where $1\le a_i\le n$ for each $1\le i\le n$.
Then $\ol{\myomega}:=[a_1,\dots,a_n]\in S_n$ and $\sum b_i=0$.
We will define the \emph{affine permutation matrix} for $\myomega$ to be the
$n\times n$ matrix $M=(m_{ij})$ with $m_{\ol{\myomega}_i,i}=z^{b_i}$ and all other entries equal to 0.
In the case where $\myomega=\ol{\myomega}\in S_n$, this reduces to the usual notion of a permutation matrix.

Given $i\not\equiv j\mod{n}$, let $t_{ij}$ denote the affine
transposition that interchanges $i$ and $j$ and leaves all $k$ not
congruent to $i$ or $j$ fixed.  Since $t_{ij}=t_{i+n,j+n}$, it
suffices to assume $1\le i\le n$ and $i<j$. 
Note that if $1\le i<j\le n$, the
above notion of transposition is the same as for the symmetric group.

As a Coxeter group, $\wt{S}_n$ is the affine Weyl group of type $A$.
It is generated by
$$S=\{s_i:=t_{i,i+1}:1\le i\le n-1\}\cup\{s_0:=t_{n,n+1}\}.$$ We will
denote the set of all affine transpositions by
$$T=\{t_{ij}:1\le i\le n, i<j, i\not\equiv j\!\!\!\!\mod{n}\}.$$ The
affine transpositions are the reflections in $\wt{S}_{n}$. 
The length function of $S_{n}$ and $\wt{S}_{n}$ can be described in
terms of one-line notation.  For a permutation $\myomega\in S_n$ one can
define an \emph{inversion} in $\myomega$ as a pair $(i,j)$ such that $i<j$
and $\myomega_i>\myomega_j$.  The length of $\myomega \in S_{n}$ is the
number of inversions in $\myomega $.  For an affine permutation, if
$\myomega_i>\myomega_j$ for some $i<j$, then we also have
$\myomega_{i+n}>\myomega_{j+n}$.  Hence, any affine permutation with a
single inversion has infinitely many inversions.  Thus, we standardize
each inversion as follows.  Define an \emph{affine inversion} as a
pair $(i,j)$ such that $1\le i \le n$, $i<j$, and $\myomega_i>\myomega_j$.
Let $\affinv(\myomega)$ denote the set of all such affine inversions.
Then we have $\ell(\myomega)=\#\affinv(\myomega)$, \cite[Proposition
8.3.1]{BB:05}.  There is also another characterization of the length
of an affine permutation due to Shi, which we won't need explicitly,
but we include for context.

\begin{prop}\textup{\cite[Lemma 4.2.2]{Shi}}
\label{length} Let $\myomega\in\wt{S}_n$.  Then
$$\ell(\myomega)=\sum_{1\le i<j\le
n}\left|\left\lfloor\frac{\myomega_j-\myomega_i}{n}\right\rfloor\right|
=\mathrm{inv}(\myomega_1,\dots,\myomega_n)+\sum_{1\le i<j\le
n}\left\lfloor\frac{|\myomega_j-\myomega_i|}{n}\right\rfloor,$$ where
$\mathrm{inv}(\myomega_1,\dots,\myomega_n)=\#\{1\le i<j\le
n:\myomega_i>\myomega_j\}$.
\end{prop}

In terms of the one-line description of the affine symmetric group, we
also have the following characterization of the covering relations for
Bruhat order.
\begin{prop}\textup{\cite[Lemma 2.2]{LLMS:08}}
\label{BruhatOrder} Let $\myomega\in\wt{S}_n$ and $i<j$.  Then $\myomega
t_{ij}\lessdot\myomega$ provided
\begin{enumerate}
\item $\myomega_i>\myomega_j$,
\item for all $i<k<j$, $\myomega_k\notin[\myomega_j,\myomega_i]$, and
\item either $j-i<n$ or $\myomega_i-\myomega_j<n$.
\end{enumerate}
\end{prop}

Define the \emph{rank function} for $\myomega\in\wt{S}_n$ by
$$r_\myomega(p,q)=\#\{i\le p:\myomega_i\ge q\}.$$ Define the
\emph{difference function} for the pair $\mychi,\myomega\in\wt{S}_n$ by
$$d_{\mychi,\myomega}(p,q)=r_\myomega(p,q)-r_\mychi(p,q).$$ The difference function gives
another useful characterization of Bruhat order and generalizes the
Ehressmann criterion for Bruhat order \cite{Ehresmann}, see also
\cite{Fulton:97}.  

\begin{prop}\textup{\cite[Theorem 8.3.7]{BB:05}}
\label{length2}
For $\mychi,\myomega\in\wt{S}_n$, $\mychi\le\myomega$ if and only if $d_{\mychi,\myomega}(p,q)\ge0$ for all $p,q\in\Z$.
\end{prop}

Since $r_\myomega(p,q)=r_\myomega(p+kn,q+kn)$ for all $k\in\Z$, it suffices to check the inequalities
in Proposition~\ref{length2} for $1\le p\le n$ and $q\in\Z$.
Note that Proposition~\ref{length2} is also valid if we replace $r_\myomega$ and $d_{\mychi,\myomega}$ with
$$r^\prime_\myomega(p,q)=\#\{i\ge p:\myomega_i\le q\}$$ and $$d^\prime_{\mychi,\myomega}(p,w)=r^\prime_\myomega(p,q)-r^\prime_\mychi(p,q).$$

\begin{cor}
\label{length2cor}
We have $\mynu\le\mychi\le\myomega$ if and only if $d_{\mynu,\myomega}(p,q)-d_{\mychi,\myomega}(p,q)\ge0$ for all $p,q\in\Z$.
\end{cor}

The following proposition characterizes the maximal proper parabolic
subgroups in $\wt{S}_n$.  Here we use the notation
$[a,b]:=\{i\in\Z:a\le i\le b\}$ for intervals among the integers.

\begin{prop}{\textup{\cite[Proposition 8.3.4]{BB:05}}}
\label{parabolic1}
Let $i\in[n]$ and $J=S\backslash\{s_i\}$.
Then $$(\wt{S}_n)_J=\Stab\left([i,i+n-1]\right)$$
and $$(\wt{S}_n)^J=\left\{\myomega\in\wt{S}_n:\myomega_1<\cdots<\myomega_i\text{ and }\myomega_{i+1}<\cdots<\myomega_{n+1}\right\}.$$
\end{prop}

\begin{cor}
\label{parabolic1cor}
If $\myomega\in\wt{S}_n$ is not in any proper parabolic subgroup, then for every $1\le i\le n$ and for every $i\le j<i+n$, there exists
some $i\le k<i+n$ such that either $\myomega_k<\myomega_j$ or $\myomega_k>\myomega_{j+n}$.
\end{cor}

\subsection{Patterns in $\wt{S}_n$}

Given a permutation $p\in S_k$ and an affine permutation
$\myomega\in\wt{S}_n$, we say that \emph{$\myomega$ contains the pattern
$p$} if there exists any subsequence of integers $i_1<\cdots<i_k$ such
that the subword $\myomega_{i_1}\cdots\myomega_{i_k}$ of $\myomega$ has the
same relative order as the elements of $p$.  We say \emph{$\myomega$
avoids the pattern $p$} if no such subsequence exists in $\myomega$.
For example, if $\myomega=[8,1,3,5,4,0]\in\wt{S}_6$, then 8,1,5,0 has
the same relative order as the pattern 4231 so $\myomega$ contains
4231, however, $\myomega$ avoids 3412.  Observe that a pattern can come
from terms outside of the base window $[\myomega_1,\dots,\myomega_n]$.
Moreover, the indices corresponding to terms in the pattern can have
the same residue mod $n$.  In the previous example, $\myomega$ also has
2,8,6 as an occurrence of the pattern 132.  Choosing a subword
$\myomega_{i_1}\cdots\myomega_{i_k}$ with the same relative order as $p$
will be referred to as \emph{placing} $p$ in $\myomega$.

We will need the following result, due to Lakshmibai and Sandhya.
Note that in type $A$, smoothness is equivalent to rational smoothness for Schubert varieties \cite{Deodhar85}.
\begin{prop}{\textup{\cite[Theorem 1]{LakSan}}}
\label{laksan}
A Schubert variety $X_\myomega\subseteq\GL_n(\C)/B$ is (rationally) smooth if and only if $\myomega$ avoids the patterns 3412 and 4231.
\end{prop}

\begin{cor}
\label{propercase}
Let $J=S\backslash\{i\}$ and suppose $\myomega\in(\wt{S}_n)_J$.
If $\myomega$ contains the pattern 4231 or 3412, then there is an occurrence of that pattern at indices $i\le i_1<i_2<i_3<i_4<i+n$,
so that $\myomega$ is not palindromic.
\end{cor}

\begin{proof}
Suppose the pattern, $p$, occurs at indices $i_1<i_2<i_3<i_4$ in $\myomega$.
By periodicity of $\myomega$ we may assume $i\le i_1<i+n$.
Since $\myomega\in(\wt{S}_n)_J$, Proposition~\ref{parabolic1} implies $\{\myomega_i,\dots,\myomega_{i+n-1}\}=\{i,\dots,i+n-1\}$.
Hence,  $\myomega_j>\myomega_{i_1}$ for all $j\ge i+n$.
Since $\myomega_{i_4}<\myomega_{i_1}$, we must have $i_4<i+n$.
Recall that $(\wt{S}_n)_J\cong S_n$.
Since $p$ is contained in the window $[\myomega_i,\dotsc,\myomega_{i+n-1}]$,
the image of $\myomega$ under this isomorphism will also contain $p$.
Hence, by Proposition~\ref{CarrellPeterson} and Proposition~\ref{laksan}, $\myomega$ is not palindromic.
\end{proof}

\subsection{Twisted Spiral Permutations}
\label{s:spirals}
For $a,b\in\Z$, define
\begin{equation}\label{e:cycles}
c(a,b)=\begin{cases}
s_as_{a+1}s_{a+2}\cdots s_{a+b-1},&\text{ if }b>0,\\
s_as_{a-1}s_{a-2}\cdots s_{a+b+1},&\text{ if }b<0,\\
1,&\text{ if }b=0,
\end{cases}
\end{equation}
where all of the subscripts are taken modulo $n$.
By definition, we have $c(a,1)=c(a,-1)=s_a$, and $\ell(c(a,b))=|b|$.
For any $k\in\Z$ with $k\not=0$ and any $1\le i\le n$, let 
\begin{equation}\label{e:spirals}
\myomega^{(i,k)}=\myomega_0^{J_i}c(i,k(n-1))\in\wt{S}_n,
\end{equation}
where $\myomega_0^{J_i}$ is the unique longest element in the parabolic subgroup generated by $J_i=S\backslash\{i\}$.
For $k \neq 0$, we call $\myomega^{(i,k)}$ a \emph{twisted spiral permutation},
since its reduced expression is obtained by spiraling around the Coxeter graph for $\wt{S}_n$ and then twisting by $\myomega_0^{J_i}$.
We could also define the twisted spiral permutations by using right cosets and twisting on the right.
This definition is equivalent because of the formula
\begin{equation}\label{e:spiralswap}
\myomega_0^{J_i}c(i,k(n-1))=c(i-kn\pm1,k(n-1))\myomega_0^{J_{i-k}},
\end{equation}
where we use $+$ if and only if $k>0$.  Thus, the set of twisted
spiral permutations is closed under taking inverses.

By a result of Mitchell~\cite{Mitchell}, the Poincar\'{e} polynomial
for $c(i,k(n-1))$ in the quotient ${}^{J_{i}}W$ is palindromic.  The
quotient ${}^{J_{i}}W$ corresponds with the affine Grassmannian of
type $A$.   To be specific, we need the following definitions for
$q$-\textit{factorial} and $q$-\textit{binomial} coefficients.  For
$m\in\N$, let
\begin{equation}\label{e:q-binomial}
(m!)_q=\prod_{k=1}^m\frac{1-q^k}{1-q}\hspace{.1in}\text{ and }\hspace{.1in}\binom{k}{j}_q=\frac{(k!)_q}{(j!)_q((k-j)!)_q}
\end{equation}
be the $q$-analogues of $m!$ and $\binom{k}{j}$.

\begin{prop}\label{p:mitchell}{\textup{\cite[Prop. 2.4 and 3.2]{Mitchell}}}
Let $J=S\backslash\{i\}$ and let $W_{J}$ be the corresponding maximal
parabolic subgroup of $W=\wt{S}_{n}$.  Then, the Poincar\'{e} for the Schubert variety $X_{c(i,k(n-1))}$ in the affine Grassmannian is
$${}^{J}P_{c(i,k(n-1))} =  \binom{k+n-1}{k}_q.$$
\end{prop}

Proposition~\ref{BilleyPostDecomp} then shows that we may lift each
$c(i,k(n-1))$ to a palindromic element in $W$ by left multiplying by
$\myomega_0^{J_i}$.  Hence the Poincar\'{e} polynomial for
$\myomega^{(i,k)}$ is palindromic.  However, each twisted spiral
permutation contains the pattern 3412.  We will see in
Section~\ref{s:3412} that elements of this infinite family,
$\left\{\myomega^{(i,k)}:k\in\Z\backslash\{0\},1\le i\le n\right\}$,
are the only palindromic permutations containing 3412.  

\begin{remark}\label{r:twisted.spiral}
For $k\geq 1$, we note that $X_{c(i,k(n-1))}$ is not a smooth variety
in the affine Grassmannian by \cite[Thm 1.1]{BilleyMitchell} since it
is not a closed parabolic orbit.  Hence, we see $X_{\myomega^{(i,k)}}$
cannot be smooth in the affine flag variety $\wt{G}/\wt{B}$. 
\end{remark}

\section{Factoring Poincar\'{e} Polynomials}
\label{One Direction}

The goal of this section is to prove the first direction of the main
theorem as stated below in Theorem~\ref{thm:onedirection}.  We will
use an adaptation of the proof for the non-affine case given by
Gasharov in \cite{Gasharov98}.  In particular, we show that if
$\myomega$ avoids 3412 and 4231, we may repeatedly factor off
palindromic terms from $\poin_\myomega(q)$.  For $\myomega\in S_n$,
Gasharov shows that the factors of the Poincar\'e polynomial for
(rationally) smooth elements are all of the form $(1+q+\cdots+q^k)$.
However, for $\myomega\in\wt{S}_n$, we will see that the factors
sometimes end up being the more general $q$-binomial coefficients
instead.

\begin{thm}
\label{thm:onedirection}
Let $\myomega\in\wt{S}_n$ be an affine permutation that avoids the patterns 3412 and 4231.
Then $P_\myomega(q)$ is palindromic.
\end{thm}

Given any finite ordered sequence of integers $t=t_1t_2\dots t_p$,
define the \emph{subsequence of left to right maxima of $t$}, denoted
$a=a_1 a_2\dots a_\ell$, as follows.  Set $a_1=t_1$ and $i_{1}=1$.  Inductively, let
$i_{j+1}$ be the smallest index greater than $i_j$ with
$t_{i_{j+1}}>t_k$ for all $k<i_{j+1}$.  Then set
$a_{j+1}=t_{i_{j+1}}$.  Note that by its construction
$t_1=a_1<\cdots<a_\ell$.  This subsequence can be obtained by reading
$t$ from left to right and recording the value each time a new maximum
is obtained.  Similarly, we can define the \emph{subsequence of right
to left maxima} of $t$, denoted $b=b_1b_2\cdots b_\ell$, by reading $t$
right to left.  Note that by its construction $b_1>\cdots>b_\ell=t_p$.

Let $\myomega\in\wt{S}_n$ and $\alpha\in\Z$.  Define the \emph{subword
of $\alpha$-inversions} of $\myomega$ by
$v^{(\alpha)}=\myomega_{i_0}\myomega_{i_1}\cdots\myomega_{i_k}$, where
$i_{0}=\alpha$ and $i_{1}<i_{2}<\dotsb <i_{k}$ indexes the subsequence
of all values to the right of $\alpha$ forming an inversion with
$\myomega_\alpha$.  Call $v^{(\alpha)}$ \emph{uninterrupted} if
$v^{(\alpha)}$ is a consecutive subword of $\myomega$.  Otherwise, we
say $v^{(\alpha)}$ is \emph{interrupted}.  We say $\myomega_\beta$
\emph{interrupts} $v^{(\alpha)}$ if $\alpha<\beta<i_k$ and
$\myomega_\beta>\myomega_\alpha$ so that $\myomega_\beta$ is not in
$v^{(\alpha)}$.

Throughout the rest of this section, assume $\myomega\in\wt{S}_n$ and $\myomega$ is
not the identity element.  Fix $\alpha\in\Z$ to be the index such that
$\myomega_\alpha$ is the largest element in the window
$[\myomega_1,\dots,\myomega_n]$ that has a non-trivial subword of
$\alpha$-inversions.  Such an $\alpha$ exists since $\myomega$ is not
the identity permutation.
Let $v=v^{(\alpha)}=\myomega_{i_0}\myomega_{i_1}\cdots\myomega_{i_p}$ be the subword of $\alpha$-inversions of $\myomega$,
and let $u=u_{1} \cdots  u_{\ell}$ be the subsequence of right to left maxima in $v$.
Let $v\backslash u$ denote all terms of $v$, which are not in $u$.
Note $u_1=v_1=\myomega_\alpha$ and $u_\ell=\myomega_{i_p}$ by construction.

Next, we prove the affine analogue of the main lemma Gasharov uses in his proof of the non-affine version of Theorem~\ref{thm:onedirection}.

\begin{lemma}
\label{Gasharov}
If $\myomega$ avoids the patterns 4231 and 3412 and $v$ is not decreasing (and hence not equal to $u$),
then $u=\myomega_\alpha(\myomega_\alpha-1)\cdots(\myomega_\alpha-m)$ for some $m$.
\end{lemma}

\begin{proof}
Since $\myomega$ avoids the pattern 4231, for any $\myomega_i\in
v\backslash u$ we must have $\myomega_i<u_\ell$.  To see this, suppose,
par contradiction, that there exists an $\myomega_i\in v\backslash u$
with $\myomega_i>u_\ell$.  Since $\myomega_i\notin u$, there must exist
some $i<j<i_p$ such that $\myomega_j\in u$ and $\myomega_i<\myomega_j$.  But
then $\myomega_\alpha\myomega_i\myomega_ju_\ell$ forms a 4231 pattern,
giving a contradiction.

Since $u\not=v$, there exists some $\myomega_j\in v\backslash u$, which
must have $\myomega_j<u_\ell$ by the above argument.  The definition of
$\alpha$ guarantees that for any integer $i<\alpha$, we have
$\myomega_i<\myomega_\alpha$.  Thus, for any $i<\alpha$ we must also have
$\myomega_i<u_\ell$ or else $\myomega_i\myomega_\alpha\myomega_ju_\ell$ will
form a 3412 pattern in $\myomega$.

Let $m=w_\alpha-u_\ell$.
Then by the previous argument and the fact that $\myomega:\Z\to\Z$ is a bijection,
every value between $\myomega_\alpha$ and $\myomega_\alpha-m$ must occur in $v$, or else $u=v$.
All of these values must occur in decreasing order, otherwise $\myomega$ would have a 4231 pattern.
\end{proof}

We now prove a lemma that is new to the affine case that arises since the subsequence $v$ need not be a consecutive subsequence.

\begin{lemma}
\label{lem:small1}
Suppose $\myomega$ avoids 3412 and 4231, and $v$ is interrupted.
If $\myomega_\beta$ is the largest element of $\myomega$ that interrupts $v$,
then $\myomega_\beta$ starts a consecutive decreasing subword ending in $u_\ell$ in the one-line notation for $\myomega$.
\end{lemma}

\begin{proof}
Suppose that $\myomega_i$ interrupts $v$ and that there exist two elements of $v$ to the right of $\myomega_i$ not in decreasing order.
Then there are $\myomega_j,\myomega_k\in v$ with $i<j<k$ and $\myomega_j<\myomega_k$.
Since $\myomega_i$ interrupts $v$ we must have $\myomega_\alpha<\myomega_i$.
But then $\myomega_\alpha\myomega_i\myomega_j\myomega_k$ will form a 3412 pattern.
Hence, all elements of $v$ to the right of the first interrupting subsequence must occur in decreasing order.

Let $\myomega_r$ be an element in the consecutive subword lying between $v_{j-1}$ and $v_j$,
and let $\myomega_s$ be an element in the consecutive subword lying between $v_{m-1}$ and $v_m$.
Suppose $j<m$ so that $r<s$.  If $\myomega_r>\myomega_s$, then,
since $v$ is decreasing to the right of the first interrupting subword,
$\myomega_r>\myomega_s>\myomega_\alpha>v_j>v_m$.
Thus, $\myomega_rv_j\myomega_sv_m$ forms a 4231 pattern, giving a contradiction.
Hence, every element in one interrupting subword is larger than
any element in any previous interrupting subword.

Let $\myomega_\beta$ be the largest element of $\myomega$ interrupting $v$, and recall $u_\ell=\myomega_{i_p}$.
By the above paragraph, $\myomega_\beta$ must occur in the last interrupting subsequence.
Then for any $\beta<j<i_p$, we must have $\myomega_j<\myomega_\beta$ since $\myomega_\beta$ was chosen to be the maximal interrupter.
We must also have $\myomega_j>\myomega_{i_p}$, since $v$ is decreasing to the right of the first interrupting subsequence.
Then all $\myomega_j$ with $\beta\le j\le i_p$ must occur in decreasing order or else we get a 4231 pattern.
\end{proof}

\begin{definition}\label{factsubword}
Let $x=x_1x_2\cdots x_k$ be a consecutive subsequence of the one-line
notation of $\myomega$.  Call $x$ a \emph{factoring subword} of $\myomega$
if it has all of the following properties:
\begin{enumerate}
\item $x$ is decreasing,\label{prop1}
\item $x_1$ is larger than all elements to its left in $\myomega$,\label{prop2}
\item $x_k$ is smaller than all elements to its right in $\myomega$,\label{prop3}
\item if $x_k$ is larger than all elements to the left of $x_1$ in $\myomega$,
	then all elements to the right of $x_k$ in $\myomega$ are larger than $x_1$.\label{prop4}
\end{enumerate}
\end{definition}
Note that Property (\ref{prop1}) implies $k\le n$ since there cannot be a
consecutive decreasing subword of length longer than $n$ in any element of $\wt{S}_n$. 

\begin{lemma}
\label{xexists}
Suppose $\myomega$ avoids the patterns 3412 and 4231.
Then either $\myomega$ or $\myomega^{-1}$ (or both) has a factoring subword.
\end{lemma}

\begin{proof}
If $v$ is decreasing and not interrupted, set $x=v$, and $x$ will be a factoring subword.
To see that Property (\ref{prop4}) holds, note that since $v$ is not interrupted, $x_1=\myomega_\alpha$,
and no element to the right of $x_k$ in $\myomega$ can form an $\alpha$-inversion.

If $v$ is not decreasing and not interrupted, then by Lemma~\ref{Gasharov},
$u=\myomega_\alpha(\myomega_\alpha-1)\cdots(\myomega_\alpha-m)$ for some $m$.
Let $x$ be the subword of $\myomega^{-1}$ corresponding to $u$.
In other words, $x=\myomega^{-1}_{\myomega_\alpha-m}\cdots\myomega^{-1}_{\myomega_\alpha-1}\myomega^{-1}_{\myomega_\alpha}$.
Note, $x$ is a consecutive subword of $\myomega^{-1}$, since the values of $u$ are all consecutive.
Property (\ref{prop1}) holds, since $u$ is decreasing.
Since $\myomega_\alpha$ is larger than all elements to its left in $\myomega$,
$\myomega^{-1}_{\myomega_\alpha}$ will be smaller than all elements to its right in $\myomega^{-1}$,
so that Property (\ref{prop2}) holds.
Similarly, since $u_\ell$ is smaller than all elements to its right in $\myomega$,
$\myomega^{-1}_{\myomega_\alpha-m}$ will be larger than all elements to its left in $\myomega^{-1}$.
Thus, Property (\ref{prop3}) holds.
Finally, if there exists some $j>\myomega_\alpha$,
such that $\alpha=\myomega^{-1}_{\myomega_\alpha}<\myomega^{-1}_j<\myomega^{-1}_{\myomega_\alpha-m}=i_p$,
then this contradicts the hypothesis that $v$ is not interrupted.
Hence, Property (\ref{prop4}) is always true.
Thus, $x$ is a factoring subword of $\myomega^{-1}$.

If $v$ is interrupted,
let $x=x_1\cdots x_k$ be the consecutive subsequence of $\myomega$
starting at the maximal interrupter, $\myomega_\beta$, and ending at $u_\ell$.
By Lemma~\ref{lem:small1}, properties (\ref{prop1})-(\ref{prop3}) hold.
Since $\myomega_\alpha$ is not in $x$ and $x_k<\myomega_\alpha$, Property (\ref{prop4}) is vacuously true.
Hence, $x$ is a factoring subword.
\end{proof}

\subsection{The Factoring Map}
\label{s:factoringmap}

The goal of this subsection is to define a map
$\Psi:\myomega\mapsto\myomega^\prime$ such that
$\poin_{\myomega^\prime}(q)$ divides $\poin_\myomega(q)$
and the quotient $\poin_\myomega(q)/\poin_{\myomega^\prime}(q)$ is palindromic.
This map will be defined whenever $\myomega$
avoids the pattern 3412 and has a factoring subword.  
 
Define $\Psi$ as follows.
Let $x=x_1\cdots x_k$ be a factoring subword for $\myomega$, and let $x_j$ be the rightmost element in $x$ such
that there is no element of $\myomega$ to the left of $x_1$ that is bigger than $x_j$.
Call $x_j$ the \emph{pivot point} of the factoring subword.
In the event where $j=k$, use $x_1$ as the pivot point instead.
This will happen, for example, when the subword of $\alpha$-inversions is not interrupted.

Suppose $x_j=\myomega_\gamma$ is the pivot point embedded in $\myomega$.
Using the definition of $c(a,b)$ in Section~\ref{s:spirals}, set $$\mysigma^{-1}=c(\gamma,k-j)c(\gamma-1,k-j)\cdots c(\gamma-(j-1),k-j).$$
Define $\Psi(\myomega)=\myomega\cdot\mysigma^{-1}$, so that $\myomega=\Psi(\myomega)\mysigma$.
In terms of the one-line notation for $\myomega$, this has the effect of shuffling $x_1\cdots x_j$ to the right past $x_{j+1}\cdots x_k$.

As an example, take $n=6$ and $$\myomega=[8,3,1,0,4,5]14,9,7,6,10,11\cdots\in\wt{S}_6.$$
Then $\myomega_\alpha=8$ so that $v=8,3,1,0,4,5,7,6$ and $u=8,7,6$.
A factoring subword is $x=14,9,7,6$, which has pivot point $x_2=\myomega_8=9$.
Hence, $$\mysigma^{-1}=c(8,2)c(7,2)=s_8s_9s_7s_8=s_2s_3s_1s_2,$$
where the indices of the generators are taken mod $n$, so that
$$\Psi(\myomega)=\myomega\cdot s_2s_3s_1s_2=[1,0,8,3,4,5]7,6,14,9,10,11\cdots.$$
\smallskip

We collect here some facts about $\mysigma$ and $\Psi(\myomega)$.  Fix
the notation $d=\gamma-(j-1)+(k-j-1)$ and $J=S\backslash\{s_d\}$.
Then, value $x_{1}$ occurs in position $d+1$ in $\Psi(\myomega)$, and
$x_{k}$ occurs in position $d$.  Since $k\le n$, $\mysigma$
corresponds to a non-affine permutation in $W_{J}$ with exactly one
descent at position $d$. By Lemma~\ref{parabolicredexp}, we
immediately get the following lemma.

\begin{lemma}
\label{reducedstart}
Every reduced expression for $\mysigma$ begins with the letter $s_d$, so that $\mysigma\in{}^J(\wt{S}_n)$.
\end{lemma}

Recall the definition of the $q$-binomial coefficients
$\binom{k}{j}_q$ in \eqref{e:q-binomial}.  Then we have the following
lemma.

\begin{lemma}
\label{sigmapalindromic}
${}^J\poin_\mysigma(q) = \binom{k}{j}_q$ which is palindromic.
\end{lemma}

\begin{proof}
Since $\mysigma$ has only one descent, the order ideal of $\mysigma$
in the quotient ${}^J(\wt{S}_n)$ corresponds to the classical
Grassmannian $G(j,k)$ consisting of all partitions that fit inside a
$(k-j)\times j$ rectangle.  In \cite[Proposition 1.3.19]{EC1} it is
shown that $\poin_\mysigma(q)=\binom{k}{j}_q$, which is known to be
palindromic by box complementation.
\end{proof}

We now show that $\Psi(\myomega)$ always lies in some proper parabolic
subgroup whenever it is defined.  Note that in the proof below we only
explicitly use avoidance of the pattern 3412.  However, the assumption
that $\myomega$ avoids the pattern 4231 is implicitly used to
guarantee the existence of a factoring subword in Lemma~\ref{xexists}.

\begin{lemma}
Suppose $\myomega$ avoids the pattern 3412 and contains a factoring subword $x=x_1\cdots x_k$. 
Set $\myomega^\prime=\Psi(\myomega)$.
Then $\myomega^\prime\in(\wt{S}_n)_J$.
\end{lemma}

\begin{proof}
Let $x_j$ be the pivot point of $x$, so that $\myomega^\prime_{d+1}=x_1$.
Let $\myomega^\prime_\delta$ be the largest element strictly to the left of $x_1$ in $\myomega^\prime$.
We will show $$\{\myomega^\prime_{d-n+1},\dots,\myomega^\prime_d\}=\{\myomega^\prime_\delta-n+1,\myomega^\prime_\delta-n+2,\dots,\myomega^\prime_\delta\}.$$

Any element in $\{\myomega^\prime_{d-n+1},\dots,\myomega^\prime_d\}$
is at most $\myomega^\prime_\delta$, since $\myomega^\prime_\delta$ was chosen to be maximal.
Suppose there is some $m$ with $0\le m\le n-1$ and $\myomega^\prime_{d-m}<\myomega^\prime_\delta-n+1$.
Then $\myomega^\prime_{d-m}<\myomega^\prime_\delta-n$,
since both $\myomega^\prime_\delta$ and $\myomega^\prime_{\delta-n}$ cannot be in the window $\{\myomega^\prime_{d-n+1},\dots,\myomega^\prime_d\}$.
Thus, $\myomega^\prime_{d-m}+n=\myomega^\prime_{d-m+n}<\myomega^\prime_\delta$.

First, suppose $\myomega^\prime_\delta$ is not in $x$, so that $\myomega^\prime_\delta>x_k$
and $\myomega^\prime_\delta$ lies to the left of $x_1$ in $\myomega$.
If $\myomega^\prime_{d-m}$ is not a value in $x$, then $\myomega^\prime_{d-m}$ lies to the left of $x_{j+1}$ in $\myomega^\prime$ and
$\myomega^\prime_{d-m+n}$ lies to the right of $x_k$ in $\myomega$.   So, $\myomega^\prime_{d-m+n}>x_k$
by Property~(\ref{prop3}) in Definition~\ref{factsubword}.
But then $\myomega^\prime_\delta x_1x_k\myomega^\prime_{d-m+n}$ forms a 3412 pattern in $\myomega$, giving a contradiction.
Otherwise, if $\myomega^\prime_{d-m}$ is in $x$, then $\myomega^\prime_{d-m}$ lies to the right of $x_1$ in $\myomega$,
so that $\myomega^\prime_\delta x_1\myomega^\prime_{d-m}\myomega^\prime_{d-m+n}$ will form a 3412 pattern in $\myomega$, giving another contradiction.

Second, suppose $\myomega^\prime_\delta$ is in $x$.
By the construction of the pivot point, this can only happen if
there are no elements to the left of $x_1$ in $\myomega$ that are larger than $x_k$.
Then $x_1=x_{j}=\myomega_\alpha$ and $x_2=\myomega_{\alpha+1}=\myomega^\prime_\delta$.
Also, $\myomega^\prime_{d-m+n}<x_2<\myomega_\alpha$, so that,
by Property~(\ref{prop4}) in Definition~\ref{factsubword}, $\myomega^\prime_{d-m+n}$ is in $x$.
But then $\myomega^\prime_{d-m}$ lies to the left of $x_1-n$ in $\myomega^\prime$,
which is a contradiction, since $\myomega^\prime_{d+1}=x_1$.
\end{proof}

\begin{lemma}
\label{max}
The element $\Psi(\myomega)=\myomega^\prime$ is maximal in $(\wt{S}_n)_J\cap[0,\myomega]$.
\end{lemma}

\begin{proof}
Since $\myomega^\prime\in(\wt{S}_n)_J$ and $\mysigma\in{}^J\wt{S}_n$, then $\myomega=\myomega^\prime\mysigma$
is the parabolic decomposition for $\myomega$ given in Proposition~\ref{decomp}.
Let $s_{i_1}\cdots s_{i_q}$ be a reduced expression for $\myomega^\prime$, 
and let $s_ds_{j_2}\cdots s_{j_r}$ be a reduced expression for $\mysigma$.
Then $s_{i_1}\cdots s_{i_q}s_ds_{j_2}\cdots s_{j_r}$ is a reduced expression for $\myomega$.
Use the construction of $m(\myomega,J)$ given in Definition~\ref{findingmwJ}.
Since $\myomega^\prime\in(\wt{S}_n)_J$, start by choosing all of the letters in $s_{i_1},\dotsc,s_{i_q}$.
Then skip $s_d$, since it is not in $J$.
Picking any letter of the reduced expression for $\mysigma$ aside from $s_d$ reduces the length of $\myomega^\prime$,
since the elements in $x_{j+1},\dots,x_k$ and $x_1,\dots,x_j$ occur in decreasing order in $\myomega^\prime$.
Hence,  $m(\myomega,J)=\myomega^\prime$.
\end{proof}

\begin{lemma}
\label{recursiveStep}
If $\myomega$ has a factoring subword and avoids the patterns 3412 and 4231, then $\Psi(\myomega)=\myomega^\prime$ also avoids 3412 and 4231.
\end{lemma}

\begin{proof}
Let $x=x_1\cdots x_k$ be a factoring subword for $\myomega$, and let $x_j=\myomega_\gamma$ be the pivot point.
Partition $x$ as $y=x_1\cdots x_j$ and $z=x_{j+1}\cdots x_k$.
Since $x$ is decreasing, this implies every element of $y$ is larger than every element of $z$.
Moreover, when considered as subwords of either $\myomega$ or $\myomega^\prime$,
every element of $y$ is larger than all elements to its left not in $y$.
Consider the following consecutive sequences in $\myomega$ and $\myomega^\prime$:
\begin{align*}
\myomega&=\cdots(y-n)\,(z-n)\left[t\,y\,z\right](t+n)\cdots,\\
\myomega^\prime&=\cdots(z-n)\,(y-n)\left[t\,z\,y\right](t+n)\cdots,
\end{align*}
where $t$ is the (possibly empty) subsequence of indices between $(z-n)$ and $y$ in $\myomega$ so $[tyz]$ forms a window of length $n$.
If $\myomega^\prime$ contains either of the patterns 4231 or 3412, we may assume the 4 of the pattern occurs in the window $[t\,z\,y]$.

First, suppose $\myomega^\prime$ has a 4231 pattern at indices $i_4<i_2<i_3<i_1$.
If $\myomega^\prime_{i_4}$ occurs in $t$ or $z$, then no other part of the 4231 pattern occurs in $y$,
or any of its translates, since all values in $y$ are bigger than all values to the left of $y$.
Hence, the relative positions of the 4231 pattern remain unchanged in $\myomega$,
so that $\myomega$ also contains 4231, giving a contradiction.
If $\myomega^\prime_{i_4}$ occurs in $y$,
then only the 2 of the 4231 pattern may occur in $y$, since $y$ is decreasing,
and no other terms of the pattern may occur in any translate of $y$.
This follows since every element of $y$ is larger than all elements to the left of $y$.
The remaining terms of the 4231 pattern occur in translates of $t$ and $z$, so that
the relative positions of the 4231 pattern remain unchanged in $\myomega$.
Hence, $\myomega$ also contains 4231, giving another contradiction.

Second, suppose $\myomega^\prime$ has a 3412 pattern at indices $i_3<i_4<i_1<i_2$.
If $\myomega^\prime_{i_4}$ occurs in $t$ or $z$, then $\myomega^\prime_{i_1}$ and $\myomega^\prime_{i_2}$ cannot occur in $y$.
Hence, they still occur to the right of $\myomega^\prime_{i_4}$ in $\myomega$ and in the correct relative order.
Since $\myomega^\prime_{i_3}$ occurs to the left of $\myomega^\prime_{i_4}$ in both  $\myomega^\prime$  and $\myomega$,
we get that $\myomega$ also contains 3412.
If $\myomega^\prime_{i_4}$ occurs in $y$,
then we may assume $\myomega^\prime_{i_4}=x_1$, since $y$ is decreasing.
Since $\myomega^\prime_{i_1}$ and $\myomega^\prime_{i_2}$ cannot occur in any translate of $y$,
they still occur to the right of $\myomega^\prime_{i_4}$ in $\myomega$ and in the correct relative order.
If $\myomega^\prime_{i_3}$ is not in $z$, then $\myomega^\prime_{i_3}$ still occurs to the left of $\myomega^\prime_{i_4}$ in $\myomega$,
so that $\myomega$ also contains 3412.
So suppose, in addition, that  $\myomega^\prime_{i_3}$ occurs in $z$.
Note that, by Property (\ref{prop3}) in Definition~\ref{factsubword}, we must have $x_k<\myomega^\prime_{i_1}<\myomega^\prime_{i_2}$.
Moreover, since $\myomega^\prime_{i_3}$ is in $z$, and $z$ is decreasing, we must have $\myomega^\prime_{i_2}<x_{j+1}$.
If there exists a value $\myomega^\prime_m$ occurring to the left of $x_{j+1}$ in $\myomega^\prime$
with $x_{j+1}<\myomega^\prime_m<x_1$, replace $\myomega^\prime_{i_3}$ with $\myomega^\prime_m$.
Then $\myomega^\prime_m$ will occur to the left of $\myomega^\prime_{i_4}$ in $\myomega$ and $\myomega^\prime_{i_2}<x_{j+1}<\myomega^\prime_m$,
so that $\myomega$ contains 3412.
Note that such a value is guaranteed to exist whenever $j>1$, by definition of the pivot point of $x$.
If $j=1$ and no such value exists, then by Property (\ref{prop4}) in Definition~\ref{factsubword},
$\myomega^\prime_{i_1}$ and $\myomega^\prime_{i_2}$ cannot exist,
contradicting the assumption that $\myomega^\prime$ contains 3412.
\end{proof}

\begin{proof}[Proof of Theorem~\ref{thm:onedirection}]
The proof is by induction on $\ell(\myomega)$.
Since $\myomega$ avoids the patterns 3412 and 4231, by Lemma~\ref{xexists}, either $\myomega$ of $\myomega^{-1}$ has a factoring subword.
We may assume $\myomega$ has a factoring subword since $\poin_\myomega(q)=\poin_{\myomega^{-1}}(q)$.
Set $\myomega^\prime=\Psi(\myomega)$, so that $\myomega=\myomega^\prime\mysigma$.
By Lemma~\ref{max}, we have $\myomega^\prime=m(\myomega,J)$.
By Lemma \ref{reducedstart}, we have $\mysigma\in{}^J(\wt{S}_n)$.
Hence, by Proposition~\ref{BilleyPostDecomp}, $$\poin_\myomega(q)=\poin_{\myomega^\prime}(q)\cdot{}^J\poin_\mysigma(q).$$

By Lemma~\ref{sigmapalindromic}, ${}^J\poin_\mysigma(q)$ is palindromic (in fact, a $q$-binomial coefficient).
Lemma~\ref{recursiveStep} shows that $\myomega^\prime$ avoids 3412 and 4231, so by induction, $\poin_{\myomega^\prime}(q)$ is also palindromic.
\end{proof}

\section{When $\myomega$ Contains 4231}
\label{Converse}
The goal of this section is to prove the following theorem.

\begin{thm}
\label{thm:4231}
Let $\myomega\in\wt{S}_n$ be an affine permutation that contains the pattern 4231.
Then $P_\myomega(q)$ is not palindromic.
\end{thm}

In particular, we show that if $$\poin_\myomega(q)=\sum_{i=0}^{\ell(\myomega)}c_iq^i,$$ then $c_1<c_{\ell(\myomega)-1}$.
Recall that $c_1$ counts the number of unique generators occurring in any reduced expression for $\myomega$,
and that $c_{\ell(\myomega)-1}$ counts the number of affine permutations covered by $\myomega$ in Bruhat order.
We will construct a graph whose edges correspond to these covering relations.
Whenever $\myomega$ contains 4231, this graph will have more than $c_1$ edges.

Let $\myomega\in\wt{S}_n$.
Fix $\beta\in\Z$ such that $\myomega_i<\myomega_\beta$ for all $i<\beta$, and let $a_1,\dots,a_r$ be the indices corresponding to the subsequence
of left to right maxima in the window $[\myomega_\beta,\dots,\myomega_{\beta+n-1}]$.
Let $G_\beta$ be the graph on the vertices 
\begin{equation}\label{e:vertices}
V_\beta=\{j\in\Z:j\ge \beta\text{ and }\myomega_j\le\myomega_{a_r}\},
\end{equation}
with edges
\begin{equation}\label{e:edges}
E_\beta=\left\{(j,k)\in V_\beta\times V_\beta:\beta\le j<\beta+n,\,j<k,\text{ and }\ell(\myomega t_{j,k})=\ell(\myomega)-1\right\}.
\end{equation}
For $1\le i\le r-1$, let $H_{i,\beta}$ be the induced subgraph of $G_\beta$ on the vertices $[a_i,a_{i+1})$,
and let $H_{r,\beta}$ be the induced subgraph of $G_\beta$ on $V_\beta\cap[a_r,\beta+n)$.
Let $H_\beta=H_{1,\beta}\cup\cdots\cup H_{r,\beta}$.

The following technical lemma is used to determine if $\myomega$ is in a proper parabolic subgroup or not.

\begin{lemma}
\label{parabolic2}
If $r\ge2$, and there exists some $1\le i\le r-1$
such that for all $a_{i+1}\le b<a_{i+1}+n$, $\myomega_b>\myomega_{a_i}$,
then $\myomega$ is in a proper parabolic subgroup.
If $\myomega$ is not in a proper parabolic subgroup, then either $r=1$,
or for all $1\le i\le r-1$ there exists some $a_{i+1}<j<a_{i+1}+n$, such that $\myomega_j<\myomega_{a_i}$.
\end{lemma}

\begin{proof}
Suppose $r\ge2$ and such an $i$ exists.
Let $\myomega_j=\min\{\myomega_{a_{i+1}},\dots,\myomega_{a_{i+1}+n-1}\}$.
Note that $\myomega_j>\myomega_{a_i}$ by our assumptions on $i$.
Also, by our choice of $j$, for all $a_{i+1}\le b<a_{i+1}+n$, we have $\myomega_j\le\myomega_b$.
Suppose there exists some $a_{i+1}\le k<a_{i+1}+n$ with $\myomega_k>\myomega_j+n$.
Then $k-n<a_{i+1}$ and $\myomega_{k-n}>\myomega_j>\myomega_{a_i}\ge\myomega_{a_1}=\myomega_\beta$.
In particular, $k-n\not=a_i$.
If $k-n<\beta$, then $\myomega_{k-n}>\myomega_\beta$, which contradicts the choice of $\beta$.
If $\beta\le k-n<a_i$, this contradicts $a_i$ being a left to right maximum.
Finally, if $a_i<k-n<a_{i+1}$, this would mean there is some left to right maximum between $a_i$ and $a_{i+1}$,
contradicting the fact that $a_i$ and $a_{i+1}$ are two consecutive maxima.
So we must have $\myomega_j\le\myomega_{a_{i+1}},\dots,\myomega_{a_{i+1}+n-1}<\myomega_j+n$.
Hence, by Proposition~\ref{parabolic1}, $\myomega$ is in a proper parabolic subgroup.
The second statement of the lemma is the contrapositive of the first.
\end{proof}

The next lemma will be useful for counting edges in $G_\beta$.

\begin{lemma}
\label{chain}
Let $\myomega\in\wt{S}_n$, and let $i<j$ be such that $\myomega_i>\myomega_j$.
If either $j-i<n$ or $\myomega_i-\myomega_j<n$, then there exists an increasing sequence $i=i_1<i_2<\cdots<i_k=j$
such that $\ell(\myomega t_{i_m,i_{m+1}})=\ell(\myomega)-1$ for each $1\le m<k$.
\end{lemma}

\begin{proof}
Suppose $j-i<n$.
Pick $i_2$ such that $\myomega_{i_1}>\myomega_{i_2}\ge\myomega_j$ and $i_2-i_1$ is minimal.
By Proposition \ref{BruhatOrder}, $\ell(\myomega t_{i_1,i_2})=\ell(\myomega)-1$.
Similarly, if we have defined $i_1<i_2<\cdots<i_{r-1}<j$, then we can pick $i_r$
such that $\myomega_{i_{r-1}}>\myomega_{i_r}\ge\myomega_j$ and $i_r-i_{r-1}$ is minimal.
Then again by Proposition \ref{BruhatOrder}, $\ell(\myomega t_{i_{r-1},i_r})=\ell(\myomega)-1$.
The same proof works when $\myomega_i-\myomega_j<n$,
only instead choose $i_r$ to minimize $\myomega_{i_{r-1}}-\myomega_{i_r}$.
\end{proof}

\begin{lemma}
\label{noassumption}
Each $H_{i,\beta}$ is connected, and hence $H_\beta$ has at least $n-r$ edges.
\end{lemma}

\begin{proof}
We start by showing that if $a_i<j<a_{i+1}$ for some $1\le i<r$, then there is a path in $G_\beta$ from $a_i$ to $j$.
By the construction of the $a_i$, we always have $a_{i+1}-a_i<n$, so that $j-a_i<n$.
Since $\myomega_{a_i}>\myomega_j$, such a path exists by Lemma~\ref{chain}.
Similarly, if $a_r\le j<\beta+n$, then there is a path in $G_\beta$ from $a_r$ to $j$.
Hence, each $H_{i,\beta}$ is connected.
Since $H_\beta$ is a graph on $n$ vertices with at most $r$ connected components, it has at least $n-r$ edges.
\end{proof}

In the case where $\myomega$ is not in a proper parabolic subgroup, we can improve on the lower bound for the number of edges in $G_\beta$.
\begin{lemma}
\label{notproper}
If $\myomega\in\wt{S}_n$ is not contained in any proper parabolic subgroup, then $G_\beta$ has at least $n$ edges.
\end{lemma}

\begin{proof}
First suppose $r=1$.
Then $\myomega_\beta$ is the largest element in the window $[\myomega_\beta,\dotsc,\myomega_{\beta+n-1}]$.  
Let $\myomega_j=\min\{\myomega_\beta,\dotsc,\myomega_{\beta+n-1}\}$.
Since $\myomega$ is not in a proper parabolic subgroup, we must have $\myomega_\beta-\myomega_{j}>n$.
There exists some $s\ge1$ such that $0<\myomega_\beta-\myomega_{j+sn}<n$.
By Lemma~\ref{chain}, there is some sequence $\beta=i_1<\cdots<i_p=j+sn$,
such that $\ell(\myomega t_{i_m,i_{m+1}})=\ell(\myomega)-1$, for each $1\le m<p$.
Let $m$ be the largest index such that $i_m$ is a vertex in $H_{1,\beta}$.
Then the edge $e_1=(i_m,i_{m+1})$ is not contained in $H_\beta$, but is in $G_\beta$.

Second, suppose $r\ge2$.
Let $1\le i<r$.
By Lemma~\ref{parabolic2}, there exists some integer $a_{i+1}<j<a_{i+1}+n$ such that $\myomega_j<\myomega_{a_i}$.
Then there exists some $s\ge0$ such that $0<\myomega_{a_i}-\myomega_{j+sn}<n$.
By Lemma~\ref{chain}, there is a sequence $a_i=i_1<\cdots<i_p=j+sn$,
such that $\ell(\myomega t_{i_m,i_{m+1}})=\ell(\myomega)-1$, for each $1\le m<p$.
Let $m$ be the largest index such that $i_m$ is a vertex in $H_{i,\beta}$.
Then the edge $e_i=(i_m,i_{m+1})$ is not contained in $H_\beta$, but is in $G_\beta$.
Repeat this process for each $1\le i<r$.

Fix $j$ such that $\myomega_j=\min\{\myomega_\beta,\dots,\myomega_{\beta+n-1}\}$.
Again, since $w$ is not in a proper parabolic subgroup,  we have some $\beta\le k<\beta+n$,
such that $\myomega_k<\myomega_j$ or $\myomega_k>\myomega_{j+n}$ by Corollary~\ref{parabolic1cor}.
But, by the definition of $\myomega_j$, we cannot have $\myomega_k<\myomega_j$.
Hence, $\myomega_k>\myomega_{j+n}$.
Since $\myomega_{a_r}$ is the largest element in $\{\myomega_\beta,\dotsc,\myomega_{\beta+n-1}\}$,
we have $\myomega_{a_r}\ge\myomega_k>\myomega_{j+n}$.
Then there exists some $s\ge1$ such that $\myomega_{a_r}-\myomega_{j+sn}<n$.
So again, by Lemma~\ref{chain}, there is a sequence $a_r=i_1<\cdots<i_p=j+sn$,
such that $\ell(\myomega t_{i_m,i_{m+1}})=\ell(\myomega)-1$, for each $1\le m<p$.
Let $m$ be the largest index such that $i_m$ is a vertex in $H_{r,\beta}$.
Then the edge $e_r=(i_m,i_{m+1})$ is not contained in $H_\beta$, but is in $G_\beta$.

We have found $r$ edges, $e_1,\dotsc,e_r$, which are not contained in $H_\beta$.
By Lemma~\ref{noassumption}, $H_\beta$ has at least $n-r$ edges.
Hence, $G_\beta$ has at least $n$ edges.
\end{proof}

\begin{proof}[Proof of Theorem~\ref{thm:4231}]
Suppose $\myomega$ contains the pattern 4231.
Then there exist indices $i_4<i_2<i_3<i_1$, such that $\myomega_{i_1}<\myomega_{i_2}<\myomega_{i_3}<\myomega_{i_4}$.
If there exists some $j<i_4$ with $\myomega_j>\myomega_{i_4}$, replace $i_4$ with $j$.
Hence, we may assume $\myomega_{i_4}>\myomega_j$, for all $j<i_4$.  
By decreasing $i_1$, $i_2$, and $i_3$,  we may assume $i_2-i_4<n$ and $i_3-i_1<n$, and also that
there is no $i_2<j<i_3$ with $\myomega_{i_2}<\myomega_j<\myomega_{i_4}$ and no $i_3<j<i_1$ with $\myomega_j<\myomega_{i_2}$.
We may then replace $i_4$ by the largest value in the range $i_4\le j<i_2$ with $\myomega_j\ge\myomega_{i_4}$.
Call the final choice of $i_4<i_2<i_3<i_1$ the normalized pattern.

Set $\beta=i_4$, and construct $G_\beta$ as above.  By construction
$a_{1}=\beta =i_{4}$.  By normalization $a_{1}<i_{2}<i_{4+n}$ and
$i_{2}<a_{2}$ if $a_{2}$ exists.

By Proposition~\ref{propercase}, we may assume $\myomega$ is not in a
proper parabolic subgroup, and hence $c_1=n$.  Note that
Lemma~\ref{notproper} gives $\left|E_\beta\right|\ge n$.  This lower
bound on the number of edges in $G_\beta$ is based on the edges in a
spanning forest of $H_{\beta}$.  So if we can show that there is
either a cycle in $H_\beta$, or another edge, different from the $e_i$
found in Lemma\ref{notproper}, that is not contained in any
$H_{i,\beta}$, we will get that $\left|E_\beta\right|>n$.

There are three ways the 4231 pattern can appear in $\myomega$ based on how the pattern intersperses with $a_{1},a_{2}$ and $i_{4}+n$.

\begin{enumerate}[\text{Case} 1:]
\item Suppose $i_1<a_2$ (or $i_1<i_4+n$ in the case where $r=1$).\\
	Then, since $a_2-a_1<n$, we have $i_1-i_4<n$.
	So, by Lemma~\ref{chain}, there is a path in $H_{1,\beta}$ connecting $i_4$ to $i_2$ and $i_3$
	and a path in $H_{1,\beta}$ connecting $i_1$ to $i_2$ and $i_3$.
	Hence,  there is a cycle in $H_{1,\beta}$.

\item Suppose $a_1<i_2<i_3<a_{2}<i_1$ (or $i_3<i_4+n<i_1$ if $r=1$).\\
	By the normalization of the 4231 pattern, we have $i_2-i_4<n$ and $i_1-i_3<n$.
	By Lemma~\ref{chain}, there is a path from $i_3$ to $i_1$ in $G_{\beta}$.
	Let $\gamma=(u,v)$ be the last edge in this path such that $u$ is a vertex in $H_{1,\beta}$.
	Also, there is some $s\ge0$ such that $0<\myomega_{i_2}-\myomega_{i_1+sn}<n$.
	Hence,  there is a sequence from $i_2$ to $i_1+sn$.
	Let $\delta=(u^\prime,v^\prime)$ be the last edge in this sequence such that $u^\prime$ is a vertex in $H_{1,\beta}$.
	
	If $u=u^\prime$, then there is a cycle in $H_{1,\beta}$, since there is already a path
	from $i_4$ to $i_2$ and a path from $i_4$ to $i_3$.
	If instead, $u\not=u^\prime$, then $\gamma\not=\delta$.
	At most one of $\gamma$ and $\delta$ can equal $e_1$ and neither are in $H_\beta$.

\item Finally, suppose $a_1<i_2<a_{2}<i_3$ (or $a_1<i_2<i_4+n<i_3$ if $r=1$).\\
	There exists some $q\ge0$ such that $0<\myomega_{i_2}-\myomega_{i_1+qn}<n$.
	Hence, there is a sequence from $i_2$ to $i_1+qn$.
	Let $\gamma=(u,v)$ be the last edge in this sequence such that $u$ is an edge in $H_{1,\beta}$.
	Similarly, there exists some $s\ge0$ such that $0<\myomega_{i_4}-\myomega_{i_3+sn}<n$.
	Hence, there is a path from $i_4$ to $i_3+sn$.
	Let $\delta=(u^\prime,v^\prime)$ be the last edge in this sequence such that $u^\prime$ is a vertex in $H_{1,\beta}$.
	Since $\myomega_v<\myomega_{i_2}<\myomega_{i_3}\le\myomega_{v^\prime},$ we have $\gamma\not=\delta$.
	At most one of $\gamma$ and $\delta$ can equal $e_1$ and neither are in $H_\beta$.
\end{enumerate}
\end{proof}

\section{Affine Bruhat Pictures}\label{s:pictures}

In this section we introduce an affine version of \textit{Bruhat
pictures}, which first appeared in \cite{BilleyWarrington}.  We use
these pictures to \textit{flatten} a pair $\mychi < \myomega$ as much
as possible while preserving the length difference and the size of the
set $\scrR(\mychi,\myomega)$.  The key result is
Corollary~\ref{cor:flatten}.

Given an affine permutation $\myomega\in\wt{S}_n$, we can visualize
$\myomega$ as $\{(i,\myomega_i):i\in\Z\}\subset\Z^2$.  We think of each
pair $(i,\myomega_{i})$ as a point in the plane drawn in Cartesian
coordinates.  Let $\pt_\myomega(i)=(i,\myomega_i)$.  Furthermore, when
comparing two affine permutations $\mychi, \myomega$ in Bruhat order using
the rank difference function $d_{\mychi,\myomega}$ and
Theorem~\ref{length2}, it is useful to visualize the nonzero entries
of the function $d_{\mychi ,\myomega}$ as a union of \textit{shaded} 
rectangles in the plane.  Combining both
visualizations we get an \textit{affine Bruhat picture}.  

For example, take $\myomega=[8,3,1,0,4,5]$ and $\mychi =\myomega
t_{1,4}=[0,3,1,8,4,5]$, then $d_{\mychi ,\myomega}$ will be positive
in the translated shaded regions of the affine Bruhat picture for
$\mychi < \myomega$ in Figure~\ref{fig:shadingexample}.  Here the
points of $\myomega$ are represented by dots, while the points of
$\mychi$ are represented by x's.

\begin{figure}
\begin{center}
\begin{tikzpicture}
\pgftransformscale{0.4}

\coordinate (ineg5) at (-5,2);
\coordinate (ineg4) at (-4,-3);
\coordinate (ineg3) at (-3,-5);
\coordinate (ineg2) at (-2,-6);
\coordinate (ineg1) at (-1,-1);
\coordinate (i0) at (0,-2);
\coordinate (i1) at (1,8);
\coordinate (i2) at (2,3);
\coordinate (i3) at (3,1);
\coordinate (i4) at (4,0);
\coordinate (i5) at (5,5);
\coordinate (i6) at (6,4);
\coordinate (i7) at (7,14);
\coordinate (i8) at (8,9);
\coordinate (i9) at (9,7);
\coordinate (i10) at (10,6);
\coordinate (i11) at (11,11);
\coordinate (i12) at (12,10);

\coordinate (jneg5) at (-5,-6);
\coordinate (jneg4) at (-4,-3);
\coordinate (jneg3) at (-3,-5);
\coordinate (jneg2) at (-2,2);
\coordinate (jneg1) at (-1,-1);
\coordinate (j0) at (0,-2);
\coordinate (j1) at (1,0);
\coordinate (j2) at (2,3);
\coordinate (j3) at (3,1);
\coordinate (j4) at (4,8);
\coordinate (j5) at (5,5);
\coordinate (j6) at (6,4);
\coordinate (j7) at (7,6);
\coordinate (j8) at (8,9);
\coordinate (j9) at (9,7);
\coordinate (j10) at (10,14);
\coordinate (j11) at (11,11);
\coordinate (j12) at (12,10);

\fill [lightgray] (jneg5) -- (ineg5) -- (jneg2) -- (ineg2) -- (jneg5);
\fill [lightgray] (j1) -- (i1) -- (j4) -- (i4) -- (j1);
\fill [lightgray] (j7) -- (i7) -- (j10) -- (i10) -- (j7);

\foreach \point in {ineg5,ineg4,ineg3,ineg2,ineg1,i0,i1,i2,i3,i4,i5,i6,i7,i8,i9,i10,i11,i12}
	{\fill [black] (\point) circle (4pt);}

\foreach \point in {jneg5,jneg4,jneg3,jneg2,jneg1,j0,j1,j2,j3,j4,j5,j6,j7,j8,j9,j10,j11,j12}
	{\draw [red] (\point) node {$\times$};}
\end{tikzpicture}
\end{center}
\caption{$\myomega=[8,3,1,0,5,4]$ and $\mychi=\myomega t_{1,4}=[0,3,1,8,5,4]$.}
\label{fig:shadingexample}
\end{figure}  

Recall from Section~\ref{sub:coxeter}, $\scrR(\mychi,\myomega) = \{t\in
T:\mychi<\mychi t\le\myomega\}$.  Given integers $p<q$ such that $\mychi_p<\mychi_q$,
we will use the affine Bruhat pictures to determine if $\mychi < \mychi
t_{p,q}\leq \myomega$.  Observe that $d_{\mychi,\mychi t_{p,q}}$ is positive
on the periodic union of rectangles
\begin{align*}
\scrA_{p,q,k}(\mychi)&=[p+kn,q-1+kn]\times[\mychi_p+1+kn,\mychi_q+kn],\text{ for all }k\in\Z.
\end{align*}
It is possible that these rectangles overlap in affine Bruhat pictures
as in Figure~\ref{fig:computingR}.  If a point $(i,j)$ is contained in exactly
$m$ consecutive translates of $\scrA_{p,q,0}(\mychi)$, then
$d_{\mychi,\mychi t_{p,q} }=m$.  Thus, we get the following criterion for
determining if $t_{p,q} \in \scrR(\mychi,\myomega)$.

\begin{lemma}
\label{shading} Given affine permutations $\mychi <\myomega$ and integers
$p<q$ such that $\mychi_p<\mychi_q$, then $t_{p,q}\in\scrR(\mychi,\myomega)$,
provided $d_{\mychi,\myomega}(i,j)\ge m$ for every $(i,j)$ contained in
$m$ consecutive rectangles $\scrA_{p,q,k}(\mychi)$.
\end{lemma}

As an example of computing $\scrR(\mychi ,\myomega )$ using shading,
let $\myomega=[6,-3,8,5,4,1]$ and $\mychi=[1,2,6,5,4,3]$ (see Figure~\ref{fig:computingR}).
Then $$\scrR(\mychi,\myomega)=\{t_{12},t_{13},t_{14},t_{15},t_{16},t_{23},t_{24},t_{25},t_{26},t_{37},t_{38},t_{47},t_{48},t_{57},t_{58},t_{67},t_{68}\}.$$

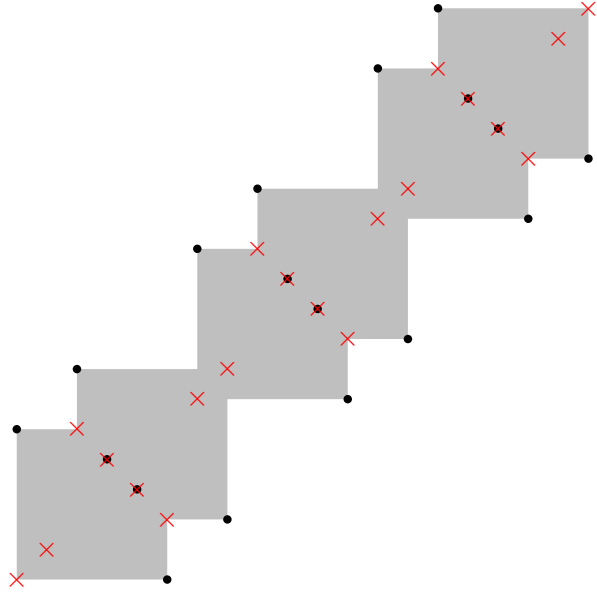
\begin{figure}
\begin{center}
\begin{tikzpicture}
\pgftransformscale{0.4}

\coordinate (ineg5) at (-5,0);
\coordinate (ineg4) at (-4,-9);
\coordinate (ineg3) at (-3,2);
\coordinate (ineg2) at (-2,-1);
\coordinate (ineg1) at (-1,-2);
\coordinate (i0) at (0,-5);
\coordinate (i1) at (1,6);
\coordinate (i2) at (2,-3);
\coordinate (i3) at (3,8);
\coordinate (i4) at (4,5);
\coordinate (i5) at (5,4);
\coordinate (i6) at (6,1);
\coordinate (i7) at (7,12);
\coordinate (i8) at (8,3);
\coordinate (i9) at (9,14);
\coordinate (i10) at (10,11);
\coordinate (i11) at (11,10);
\coordinate (i12) at (12,7);
\coordinate (i14) at (14,9);

\coordinate (jneg5) at (-5,-5);
\coordinate (jneg4) at (-4,-4);
\coordinate (jneg3) at (-3,0);
\coordinate (jneg2) at (-2,-1);
\coordinate (jneg1) at (-1,-2);
\coordinate (j0) at (0,-3);
\coordinate (j1) at (1,1);
\coordinate (j2) at (2,2);
\coordinate (j3) at (3,6);
\coordinate (j4) at (4,5);
\coordinate (j5) at (5,4);
\coordinate (j6) at (6,3);
\coordinate (j7) at (7,7);
\coordinate (j8) at (8,8);
\coordinate (j9) at (9,12);
\coordinate (j10) at (10,11);
\coordinate (j11) at (11,10);
\coordinate (j12) at (12,9);
\coordinate (j13) at (13,13);
\coordinate (j14) at (14,14);

\fill [lightgray] (i1) -- (j3) -- (i3) -- (j8) -- (i8) -- (j6) -- (i6) -- (j1) -- (i1);
\fill [lightgray] (ineg5) -- (jneg3) -- (ineg3) -- (j2) -- (i2) -- (j0) -- (i0) -- (jneg5) -- (ineg5);
\fill [lightgray] (i7) -- (j9) -- (i9) -- (j14) -- (i14) -- (j12) -- (i12) -- (j7) -- (i7);

\foreach \point in {ineg5,ineg3,ineg2,ineg1,i0,i1,i2,i3,i4,i5,i6,i7,i8,i9,i10,i11,i12,i14}
	{\fill [black] (\point) circle (4pt);}

\foreach \point in {jneg5,jneg4,jneg3,jneg2,jneg1,j0,j1,j2,j3,j4,j5,j6,j7,j8,j9,j10,j11,j12,j13,j14}
	{\draw [red] (\point) node {$\times$};}
\end{tikzpicture}
\end{center}
\caption{$\myomega=[6,-3,8,5,4,1]$ and $\mychi=[1,2,6,5,4,3]$.}
\label{fig:computingR}
\end{figure}  

\subsection{Flattening Affine Permutations}\label{sub:flatten}

Whenever $\mychi_i=\myomega_i$, the function $d_{\mychi ,\myomega }$ and the
Bruhat pictures would remain the same if we removed the entire column
$i$ and row $\mychi_i=\myomega_i$ from the plane.  However, removing some of
these rows and columns would change $\ell(\myomega)-\ell(\mychi)$ and/or
$\# \scrR(\mychi,\myomega)$.  We want to identify which rows and columns
where $\mychi_i=\myomega_i$ can be removed while maintaining the length
difference and $\# \scrR(\mychi,\myomega)$.
This will simplify the Bruhat pictures that follow.

Consider $\myomega \in \wt{S}_n$ as points drawn in the plane.  For $i
\in \mathbb{Z}$, let $\myomega^{\wh{i}}$ be the unique affine
permutation with points obtained from the points of $\myomega$ by
deleting columns $i+kn$ and rows $\myomega_{i}+kn$ for all $k \in
\mathbb{Z}$.  The remaining rows and columns of this infinite array
should now be relabeled in such a way that
$\myomega^{\wh{i}}_{1}+\myomega^{\wh{i}}_{2}+\ldots
+\myomega^{\wh{i}}_{n-1}=\binom{n}{2}$.  This can be done in a unique way since every
window of $n-1$ distinct values mod $n-1$ determines a periodic
function from $\mathbb{Z}$ to $\mathbb{Z}$ with period $n-1$ and some
translate of this window must sum to $\binom{n}{2}$.

\begin{lemma}\label{l:length.difference}
Suppose $\mychi<\myomega\in\wt{S}_n$.  Let $i$ be an integer such that
$\pt_\mychi (i)=\pt_\myomega(i)$ and
$d_{\mychi,\myomega}(\pt_\mychi(i))=d^\prime_{\mychi,\myomega}(\pt_\mychi(i))=0$.
Then,
$\ell(\myomega)-\ell(\mychi)=\ell(\myomega^{\wh{i}})-\ell(\mychi^{\wh{i}})$.
\end{lemma}

\begin{proof}
Recall, $\ell(\myomega)$ is the number of affine inversions for $\myomega$.  We
claim $\ell(\myomega)-\ell(\myomega^{\wh{i}}) =
r_\myomega(\pt_\myomega(i))+r^\prime_\myomega(\pt_\myomega(i))$ since all the
affine inversions which involve a translate of $\pt_\myomega(i)$ can be
represented by a point northwest of $\pt_\myomega(i)$ or southeast of
$\pt_\myomega(i)$.  Therefore,
\begin{align*}
\ell(\myomega)-\ell(\mychi)-(\ell(\myomega^{\wh{i}})-\ell(\mychi^{\wh{i}}))&=\left(\ell(\myomega)-\ell(\myomega^{\wh{i}})\right)-\left(\ell(\mychi)-\ell(\mychi^{\wh{i}})\right)\\
&=r_\myomega(\pt_\myomega(i))+r^\prime_\myomega(\pt_\myomega(i))-\left(r_\mychi(\pt_\mychi(i))+r^\prime_\mychi(\pt_\mychi(i))\right)\\
&=d_{\mychi,\myomega}(\pt_\mychi(i))+d^\prime_{\mychi,\myomega}(\pt_\mychi(i)).
\end{align*}
The last line is 0 by assumption, thus proving the lemma.
\end{proof}

\begin{lemma}\label{l:R.bijection}
Suppose $\mychi<\myomega\in\wt{S}_n$.  Let $i$ be an integer such that
$\mychi_i=\myomega_i$ and $d(i,\myomega_i+1)=d(i-1,\myomega_i)=0$.
Then there is a bijection between
$\scrR(\mychi,\myomega)$ and $\scrR(\mychi^{\wh{i}},\myomega^{\wh{i}}).$
The  analogous statement holds if, instead, $d^\prime(i,\myomega_i-1)=d^\prime(i+1,\myomega_i)=0$.
\end{lemma}

\begin{proof}
Consider the rows and columns of $\myomega^{\wh{i}}$ to be labeled as a
subset of the rows and columns of $\myomega$ instead of relabeling.
Since $\mychi_{i} =\myomega_{i}$, the values of $d_{\mychi,\myomega}$ and
$d_{\mychi^{\wh{i}},\myomega^{\wh{i}}}$ agree at all the points where $d_{\mychi^{\wh{i}},\myomega^{\wh{i}}}$ is defined.
Hence, by Lemma~\ref{shading}, for every $t_{a,b}\in\scrR(\mychi^{\wh{i}},\myomega^{\wh{i}})$,
there is a corresponding reflection in $\scrR(\mychi,\myomega)$.

Conversely, we claim every $t_{a,b} \in \scrR(\mychi,\myomega)$ has
$a\not\equiv i$ and $b\not\equiv i\pmod{n}$ and so corresponds with a
reflection in $\scrR(\mychi^{\wh{i}},\myomega^{\wh{i}})$, by Lemma~\ref{shading} again.  
To prove the claim, assume that there exists a $t_{i,b} \in \scrR(\mychi,\myomega)$ with $i<b$.
By hypothesis, $d_{\mychi,\myomega}(i,\myomega_{i}+1)=0$, so $d_{\mychi t_{i,b},\myomega}(i,\myomega_{i}+1)<0$,
which implies $\mychi t_{i,b} \not\leq \myomega$, contradicting the assumption that $t_{i,b}\in\scrR(\mychi,\myomega)$.
Similarly, since $d_{\mychi,\myomega}(i-1,\myomega_{i})=0$, no $t_{a,i}\in\scrR(\mychi,\myomega)$ exists for $a<i$.
The same proof works if, instead, $d^\prime(i,\myomega_i-1)=d^\prime(i+1,\myomega_i)=0$.
\end{proof}

We say the pair $\mychi <\myomega$ is \emph{flattened} if no index
$i$ exists such that 
\begin{enumerate}
\item $\mychi_i=\myomega_i$,
\item $d_{\mychi,\myomega}(\pt_\mychi(i))=d^\prime_{\mychi,\myomega}(\pt_\mychi(i))=0$, and
\item $d(i,\myomega_i+1)=d(i-1,\myomega_i)=0$\ or \ $d^\prime(i,\myomega_i-1)=d^\prime(i+1,\myomega_i)=0$.
\end{enumerate}
Otherwise, the pair is \emph{flattenable}.  Observe that
$\mychi^{\wh{i}}<\myomega^{\wh{i}}$ since $\mychi <\myomega$ and the values of
$d_{\mychi,\myomega}$ and $d_{\mychi^{\wh{i}},\myomega^{\wh{i}}}$ agree at all
the points where $d_{\mychi^{\wh{i}},\myomega^{\wh{i}}}$ is defined.  Let
$\wt{\mychi } <\wt{\myomega}$ be the \textit{flattened} pair obtained from
$\mychi <\myomega$ by iteratively removing every possible index until the
pair is flattened.  Then by Lemma~\ref{l:length.difference} and
Lemma~\ref{l:R.bijection} we have the following corollary.

\begin{cor}\label{cor:flatten}
Suppose $\mychi<\myomega\in\wt{S}_n$.  Then
$\#\scrR(\mychi,\myomega)=\ell(\myomega)-\ell(\mychi)$ if and only if
$\#\scrR(\wt{\mychi},\wt{\myomega})=\ell(\wt{\myomega})-\ell(\wt{\mychi})$.
\end{cor}

\section{When $\myomega$ Contains 3412}
\label{s:3412}

The goal of this section is to prove that any affine permutation
$\myomega$ containing a 3412 pattern is either a twisted spiral element or there
exists a $\mychi < \myomega $ such that $\# \scrR(\mychi,\myomega) >
\ell(\myomega).$   In the latter case, this implies $\poin_\myomega(q)$ is not
palindromic and $X_\myomega$ is not rationally smooth by
Theorem~\ref{CarrellPeterson}.    We simplify the proofs by squeezing the
pattern down as much as possible to a \textit{normalized} 3412 pattern
and \textit{flattening} $\mychi,\myomega$ down to the smallest
$\wt{S}_{n}$ possible.  The key tool we use is the affine Bruhat pictures defined in Section~\ref{s:pictures}. 
The arguments generalize \cite[Section 4]{BilleyWarrington} to affine permutations.  

\begin{thm}
\label{thm:3412}
If $\myomega\in\wt{S}_n$ contains the pattern 3412, but is not a twisted spiral permutation, then $\poin_\myomega(q)$ is not palindromic.
\end{thm}

\begin{proof}
Suppose $\myomega\in\wt{S}_n$ contains 3412 at indices
$i_3<i_4<i_1<i_2$.  By Theorem~\ref{thm:4231} we may also assume that
$\myomega$ avoids the pattern 4231.
We normalize the pattern that occurs by squeezing it together as far
as possible.  This is accomplished using the following procedure.

\begin{enumerate}
\item Let $i_4^\prime<i_1$ be the maximal index such that there exists some
	$j<i_4^\prime$ with $\myomega_{i_4^\prime}>\myomega_j>\myomega_{i_{2}}$.
\item Let $i_1^\prime>i_4^\prime$ be the minimal index such that there exists some
	$k>i_1^\prime$ with $\myomega_{i_1^\prime}<\myomega_k<\myomega_{j}$.  
\item  Let $i_3^\prime<i_4^\prime$ be the maximal index such that
	$\myomega_{k}<\myomega_{i_3^\prime}<\myomega_{i_4^\prime}$.  
\item Let $i_2^\prime>i_1^\prime$ be the minimal index such that
	$\myomega_{i_1^\prime}<\myomega_{i_2^\prime}<\myomega_{i_3^\prime}$.  
\end{enumerate}

\noindent
Keep repeating this process until no more changes occur.
Since $0<i^\prime_1-i^\prime_4<i_1-i_4$, this process will eventually terminate.
Renaming $i_j^\prime$ to $i_j$ gives a 3412 pattern as in Figure~\ref{fig:normalize},
where the dotted regions contains no points in $\myomega$ including the unbounded
regions above and below the pattern.  Furthermore, since $\myomega$ avoids the
pattern 4231, any values of $\myomega$ lying in regions $A$, $B$, or $C$
must occur in decreasing order.

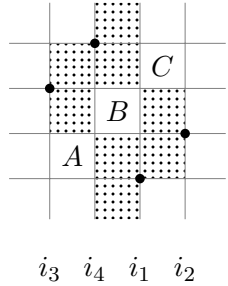
\begin{figure}[h]
\begin{center}
\begin{tikzpicture}
\pgftransformscale{0.6}

\coordinate (i3) at (0,2);
\coordinate (i4) at (1,3);
\coordinate (i1) at (2,0);
\coordinate (i2) at (3,1);

\fill [pattern=dots] (0,1) -- (0,3) -- (1,3) -- (1,3.9) -- (2,3.9) -- (2,2) -- (1,2) -- (1,1) -- (0,1);
\fill [pattern=dots] (1,-0.9) -- (1,1) -- (2,1) -- (2,2) -- (3,2) -- (3,0) -- (2,0) -- (2,-0.9) -- (1,-0.9);
\draw[step=1,gray,very thin] (-0.9,-0.9) grid (3.9,3.9);
\foreach \point in {i3,i4,i1,i2}
	{\fill [black] (\point) circle (3pt);}

\draw (0,-2) node{$i_3$};
\draw (1,-2) node{$i_4$};
\draw (2,-2) node{$i_1$};
\draw (3,-2) node{$i_2$};
\draw (0.5,0.5) node{$A$};
\draw (1.5,1.5) node{$B$};
\draw (2.5,2.5) node{$C$};

\end{tikzpicture}
\end{center}
\caption{Normalization of the $3412$ pattern.}
\label{fig:normalize}
\end{figure}

The proof now proceeds by cases
depending on whether or not region $B$ is empty and the number of
congruence classes among $\{i_{1},i_{2},i_{3},i_{4} \}$ mod $n$.  Note
that $i_{1} \equiv i_{2}$ and $i_{3}\equiv i_{4}$ are the only
possible congruences among these indices, since they index a 3412
pattern.

\subsection{When Region $B$ is Nonempty.}  By the normalization of the
3412 pattern, we must have
$\myomega_{i_3}>\myomega_{i_4+1}>\cdots>\myomega_{i_4+k}>\myomega_{i_2}$,
where $k=i_1-i_4-1\ge1$.  We will refer to such a pattern as a
$45312_k$ pattern.  Regions $A$ and $C$ must be empty or else
$\myomega$ contains a 4231 pattern.  Thus, we have the picture as in
Figure~\ref{fig:normalize45312}, where the dotted regions are empty,
including the unbounded regions on all 4 sides.  

\begin{figure}[h]
\begin{center}
\begin{tikzpicture}
\pgftransformscale{0.6}

\coordinate (i4) at (0,4);
\coordinate (i5) at (1,5);
\coordinate (a1) at (2,3);
\coordinate (ak) at (3,2);
\coordinate (i1) at (4,0);
\coordinate (i2) at (5,1);

\fill [pattern=dots] (0,0) -- (0,1) -- (-0.9,1) -- (-0.9,3) -- (0,3) -- (0,5) -- (1,5) -- (1,5.9) -- (4,5.9) -- (4,5) -- (5,5) -- (5,4) -- (5.9, 4) -- (5.9,2)
	-- (5,2) -- (5,0) -- (4,0) -- (4,-0.9) -- (1,-0.9) -- (1,0) -- (0,0);
\draw[step=1,gray,very thin] (-0.9,-0.9) grid (5.9,5.9);
\foreach \point in {i4,i5,a1,ak,i1,i2}
	{\fill [black] (\point) circle (3pt);}

\draw (0,-2) node{$i_3$};
\draw (1,-2) node{$i_4$};
\draw (4,-2) node{$i_1$};
\draw (5,-2) node{$i_2$};

\draw[loosely dotted,thick] (a1) -- (ak);

\end{tikzpicture}
\end{center}
\caption{Normalized $45312_k$ pattern.}
\label{fig:normalize45312}
\end{figure}
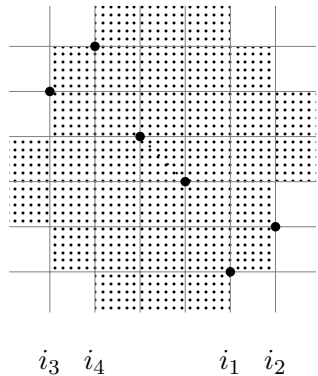

\begin{enumerate}[\text{Case} 1:]
\item All residues are distinct among $\{i_1,i_2,i_3,i_4\}$.\\
	Let $\mychi=\myomega t_{i_3i_2}t_{i_4i_2}t_{i_3i_1}$.
	Based on Corollary~\ref{cor:flatten} we can replace $\mychi$ with $\wt{\mychi}$ and $\myomega$ with $\wt{\myomega}$.
	Hence, we may assume that $n=k+4$ and that the indices $\{i_1,i_2,i_3,i_4,i_4+1,\dotsc,i_4+k\}$
	are all of the distinct residues mod $n$.
	The next step is to consider how the translated values of
	$\myomega_{i_1},\myomega_{i_2},\myomega_{i_3},\myomega_{i_4},\myomega_{i_4+1},\dots,\myomega_{i_4+k}$
	intersperse the $45312_k$ pattern.
	
	We claim that, assuming the 3412 pattern is normalized, the pair $\mychi,\myomega$ is flattened, and $\myomega$ does not contain 4231,
	there are only 4 possible interspersed arrangements as shown in Figure~\ref{fig:alldistinct}.
	Since the pair $\mychi,\myomega$ is flattened, we can assume there are no
	points in the affine Bruhat picture outside the shaded regions.
	Examining the affine Bruhat pictures in Figure~\ref{fig:alldistinct}, we see that, in the first picture,
	$\ell(\myomega)-\ell(\mychi)=2k+3$ and $\#\scrR(\mychi,\myomega)\ge2k+4$.
	In the remaining three pictures, we have
	$\ell(\myomega)-\ell(\mychi)=2k+5$ and $\#\scrR(\mychi,\myomega)>2k+5$.
	Hence,  in all of these cases, $\poin_\myomega(q)$ cannot be palindromic by Proposition~\ref{CarrellPeterson}.
	
	\begin{figure}
	\begin{center}
	\begin{displaymath}
	\begin{array}{c|c}
	\begin{tikzpicture}
	\pgftransformscale{0.3}

	\coordinate (ineg5) at (-5,-1);
	\coordinate (ineg4) at (-4,0);
	\coordinate (ineg3) at (-3,-2);
	\coordinate (ineg2) at (-2,-3);
	\coordinate (ineg1) at (-1,-5);
	\coordinate (i0) at (0,-4);
	\coordinate (i1) at (1,5);
	\coordinate (i2) at (2,6);
	\coordinate (i3) at (3,4);
	\coordinate (i4) at (4,3);
	\coordinate (i5) at (5,1);
	\coordinate (i6) at (6,2);
	\coordinate (i7) at (7,11);
	\coordinate (i8) at (8,12);
	\coordinate (i9) at (9,10);
	\coordinate (i10) at (10,9);
	\coordinate (i11) at (11,7);
	\coordinate (i12) at (12,8);

	\coordinate (jneg5) at (-5,-5);
	\coordinate (jneg4) at (-4,-1);
	\coordinate (jneg3) at (-3,-2);
	\coordinate (jneg2) at (-2,-3);
	\coordinate (jneg1) at (-1,-4);
	\coordinate (j0) at (0,0);
	\coordinate (j1) at (1,1);
	\coordinate (j2) at (2,5);
	\coordinate (j3) at (3,4);
	\coordinate (j4) at (4,3);
	\coordinate (j5) at (5,2);
	\coordinate (j6) at (6,6);
	\coordinate (j7) at (7,7);
	\coordinate (j8) at (8,11);
	\coordinate (j9) at (9,10);
	\coordinate (j10) at (10,9);
	\coordinate (j11) at (11,8);
	\coordinate (j12) at (12,12);

	\fill [lightgray] (j1) -- (i1) -- (j2) -- (i2) -- (j6) -- (i6) -- (j5) -- (i5) -- (j1);
	\fill [lightgray] (jneg5) -- (ineg5) -- (jneg4) -- (ineg4) -- (j0) -- (i0) -- (jneg1) -- (ineg1) -- (jneg5);
	\fill [lightgray] (j7) -- (i7) -- (j8) -- (i8) -- (j12) -- (i12) -- (j11) -- (i11) -- (j7);

	\draw (j1) -- (i1) -- (j2) -- (i2) -- (j6) -- (i6) -- (j5) -- (i5) -- (j1);
	\draw (jneg5) -- (ineg5) -- (jneg4) -- (ineg4) -- (j0) -- (i0) -- (jneg1) -- (ineg1) -- (jneg5);
	\draw (j7) -- (i7) -- (j8) -- (i8) -- (j12) -- (i12) -- (j11) -- (i11) -- (j7);

	\foreach \point in {ineg5,ineg4,ineg3,ineg2,ineg1,i0,i1,i2,i3,i4,i5,i6,i7,i8,i9,i10,i11,i12}
		{\fill [black] (\point) circle (4pt);}

	\foreach \point in {jneg5,jneg4,jneg3,jneg2,jneg1,j0,j1,j2,j3,j4,j5,j6,j7,j8,j9,j10,j11,j12}
		{\draw [red] (\point) node {$\times$};}
		
	\draw[loosely dotted,thick] (i3) -- (i4);
	\draw[loosely dotted,thick] (ineg3) -- (ineg2);
	\draw[loosely dotted,thick] (i9) -- (i10);
	\end{tikzpicture}
	&
	\begin{tikzpicture}
	\pgftransformscale{0.3}

	\coordinate (ineg5) at (-5,0);
	\coordinate (ineg4) at (-4,-9);
	\coordinate (ineg3) at (-3,2);
	\coordinate (ineg2) at (-2,-1);
	\coordinate (ineg1) at (-1,-2);
	\coordinate (i0) at (0,-5);
	\coordinate (i1) at (1,6);
	\coordinate (i2) at (2,-3);
	\coordinate (i3) at (3,8);
	\coordinate (i4) at (4,5);
	\coordinate (i5) at (5,4);
	\coordinate (i6) at (6,1);
	\coordinate (i7) at (7,12);
	\coordinate (i8) at (8,3);
	\coordinate (i9) at (9,14);
	\coordinate (i10) at (10,11);
	\coordinate (i11) at (11,10);
	\coordinate (i12) at (12,7);
	\coordinate (i14) at (14,9);

	\coordinate (jneg5) at (-5,-5);
	\coordinate (jneg4) at (-4,-4);
	\coordinate (jneg3) at (-3,0);
	\coordinate (jneg2) at (-2,-1);
	\coordinate (jneg1) at (-1,-2);
	\coordinate (j0) at (0,-3);
	\coordinate (j1) at (1,1);
	\coordinate (j2) at (2,2);
	\coordinate (j3) at (3,6);
	\coordinate (j4) at (4,5);
	\coordinate (j5) at (5,4);
	\coordinate (j6) at (6,3);
	\coordinate (j7) at (7,7);
	\coordinate (j8) at (8,8);
	\coordinate (j9) at (9,12);
	\coordinate (j10) at (10,11);
	\coordinate (j11) at (11,10);
	\coordinate (j12) at (12,9);
	\coordinate (j13) at (13,13);
	\coordinate (j14) at (14,14);

	\fill [lightgray] (j1) -- (i1) -- (j3) -- (i3) -- (j8) -- (i8) -- (j6) -- (i6) -- (j1);
	\fill [lightgray] (jneg5) -- (ineg5) -- (jneg3) -- (ineg3) -- (j2) -- (i2) -- (j0) -- (i0) -- (jneg5);
	\fill [lightgray] (j7) -- (i7) -- (j9) -- (i9) -- (j14) -- (i14) -- (j12) -- (i12) -- (j7);

	\draw (j1) -- (i1) -- (j3) -- (i3) -- (j8) -- (i8) -- (j6) -- (i6) -- (j1);
	\draw (jneg5) -- (ineg5) -- (jneg3) -- (ineg3) -- (j2) -- (i2) -- (j0) -- (i0) -- (jneg5);
	\draw (j7) -- (i7) -- (j9) -- (i9) -- (j14) -- (i14) -- (j12) -- (i12) -- (j7);

	\foreach \point in {ineg5,ineg3,ineg2,ineg1,i0,i1,i2,i3,i4,i5,i6,i7,i8,i9,i10,i11,i12,i14}
		{\fill [black] (\point) circle (4pt);}

	\foreach \point in {jneg5,jneg4,jneg3,jneg2,jneg1,j0,j1,j2,j3,j4,j5,j6,j7,j8,j9,j10,j11,j12,j13,j14}
		{\draw [red] (\point) node {$\times$};}
		
	\draw[loosely dotted,thick] (i4) -- (i5);
	\draw[loosely dotted,thick] (ineg2) -- (ineg1);
	\draw[loosely dotted,thick] (i10) -- (i11);
	\end{tikzpicture}
	\\\hline
	\begin{tikzpicture}
	\pgftransformscale{0.3}

	\coordinate (ineg5) at (-5,0);
	\coordinate (ineg4) at (-4,-10);
	\coordinate (ineg3) at (-3,5);
	\coordinate (ineg2) at (-2,-2);
	\coordinate (ineg1) at (-1,-3);
	\coordinate (i0) at (0,-5);
	\coordinate (i1) at (1,6);
	\coordinate (i2) at (2,-4);
	\coordinate (i3) at (3,11);
	\coordinate (i4) at (4,4);
	\coordinate (i5) at (5,3);
	\coordinate (i6) at (6,1);
	\coordinate (i7) at (7,12);
	\coordinate (i8) at (8,2);
	\coordinate (i9) at (9,17);
	\coordinate (i10) at (10,10);
	\coordinate (i11) at (11,9);
	\coordinate (i12) at (12,7);
	\coordinate (i14) at (14,8);

	\coordinate (jneg5) at (-5,-5);
	\coordinate (jneg4) at (-4,-1);
	\coordinate (jneg3) at (-3,0);
	\coordinate (jneg2) at (-2,-2);
	\coordinate (jneg1) at (-1,-3);
	\coordinate (j0) at (0,-4);
	\coordinate (j1) at (1,1);
	\coordinate (j2) at (2,5);
	\coordinate (j3) at (3,6);
	\coordinate (j4) at (4,4);
	\coordinate (j5) at (5,3);
	\coordinate (j6) at (6,2);
	\coordinate (j7) at (7,7);
	\coordinate (j8) at (8,11);
	\coordinate (j9) at (9,12);
	\coordinate (j10) at (10,10);
	\coordinate (j11) at (11,9);
	\coordinate (j12) at (12,8);
	\coordinate (j13) at (13,13);
	\coordinate (j14) at (14,17);

	\fill [lightgray] (j1) -- (i1) -- (j3) -- (i3) -- (j8) -- (i8) -- (j6) -- (i6) -- (j1);
	\fill [lightgray] (jneg5) -- (ineg5) -- (jneg3) -- (ineg3) -- (j2) -- (i2) -- (j0) -- (i0) -- (jneg5);
	\fill [lightgray] (j7) -- (i7) -- (j9) -- (i9) -- (j14) -- (i14) -- (j12) -- (i12) -- (j7);

	\draw (j1) -- (i1) -- (j3) -- (i3) -- (j8) -- (i8) -- (j6) -- (i6) -- (j1);
	\draw (jneg5) -- (ineg5) -- (jneg3) -- (ineg3) -- (j2) -- (i2) -- (j0) -- (i0) -- (jneg5);
	\draw (j7) -- (i7) -- (j9) -- (i9) -- (j14) -- (i14) -- (j12) -- (i12) -- (j7);

	\foreach \point in {ineg5,ineg3,ineg2,ineg1,i0,i1,i2,i3,i4,i5,i6,i7,i8,i9,i10,i11,i12,i14}
		{\fill [black] (\point) circle (4pt);}

	\foreach \point in {jneg5,jneg4,jneg3,jneg2,jneg1,j0,j1,j2,j3,j4,j5,j6,j7,j8,j9,j10,j11,j12,j13,j14}
		{\draw [red] (\point) node {$\times$};}

	\draw[loosely dotted,thick] (i4) -- (i5);
	\draw[loosely dotted,thick] (ineg2) -- (ineg1);
	\draw[loosely dotted,thick] (i10) -- (i11);
	\end{tikzpicture}
	&
	\begin{tikzpicture}
	\pgftransformscale{0.3}

	\coordinate (ineg5) at (-5,1);
	\coordinate (ineg4) at (-4,-9);
	\coordinate (ineg3) at (-3,2);
	\coordinate (ineg2) at (-2,0);
	\coordinate (ineg1) at (-1,-1);
	\coordinate (i0) at (0,-8);
	\coordinate (i1) at (1,7);
	\coordinate (i2) at (2,-3);
	\coordinate (i3) at (3,8);
	\coordinate (i4) at (4,6);
	\coordinate (i5) at (5,5);
	\coordinate (i6) at (6,-2);
	\coordinate (i7) at (7,13);
	\coordinate (i8) at (8,3);
	\coordinate (i9) at (9,14);
	\coordinate (i10) at (10,12);
	\coordinate (i11) at (11,11);
	\coordinate (i12) at (12,4);
	\coordinate (i14) at (14,9);

	\coordinate (jneg5) at (-5,-8);
	\coordinate (jneg4) at (-4,-4);
	\coordinate (jneg3) at (-3,1);
	\coordinate (jneg2) at (-2,0);
	\coordinate (jneg1) at (-1,-1);
	\coordinate (j0) at (0,-3);
	\coordinate (j1) at (1,-2);
	\coordinate (j2) at (2,2);
	\coordinate (j3) at (3,7);
	\coordinate (j4) at (4,6);
	\coordinate (j5) at (5,5);
	\coordinate (j6) at (6,3);
	\coordinate (j7) at (7,4);
	\coordinate (j8) at (8,8);
	\coordinate (j9) at (9,13);
	\coordinate (j10) at (10,12);
	\coordinate (j11) at (11,11);
	\coordinate (j12) at (12,9);
	\coordinate (j13) at (13,10);
	\coordinate (j14) at (14,14);

	\fill [lightgray] (j1) -- (i1) -- (j3) -- (i3) -- (j8) -- (i8) -- (j6) -- (i6) -- (j1);
	\fill [lightgray] (jneg5) -- (ineg5) -- (jneg3) -- (ineg3) -- (j2) -- (i2) -- (j0) -- (i0) -- (jneg5);
	\fill [lightgray] (j7) -- (i7) -- (j9) -- (i9) -- (j14) -- (i14) -- (j12) -- (i12) -- (j7);

	\draw (j1) -- (i1) -- (j3) -- (i3) -- (j8) -- (i8) -- (j6) -- (i6) -- (j1);
	\draw (jneg5) -- (ineg5) -- (jneg3) -- (ineg3) -- (j2) -- (i2) -- (j0) -- (i0) -- (jneg5);
	\draw (j7) -- (i7) -- (j9) -- (i9) -- (j14) -- (i14) -- (j12) -- (i12) -- (j7);

	\foreach \point in {ineg5,ineg3,ineg2,ineg1,i0,i1,i2,i3,i4,i5,i6,i7,i8,i9,i10,i11,i12,i14}
		{\fill [black] (\point) circle (4pt);}

	\foreach \point in {jneg5,jneg4,jneg3,jneg2,jneg1,j0,j1,j2,j3,j4,j5,j6,j7,j8,j9,j10,j11,j12,j13,j14}
		{\draw [red] (\point) node {$\times$};}

	\draw[loosely dotted,thick] (i4) -- (i5);
	\draw[loosely dotted,thick] (ineg2) -- (ineg1);
	\draw[loosely dotted,thick] (i10) -- (i11);
	\end{tikzpicture}
	\end{array}
	\end{displaymath}
	\end{center}
	\caption{Cases with $i_3\not\equiv i_4$ and $i_1\not\equiv i_2\pmod{n}$.}
	\label{fig:alldistinct}
	\end{figure}
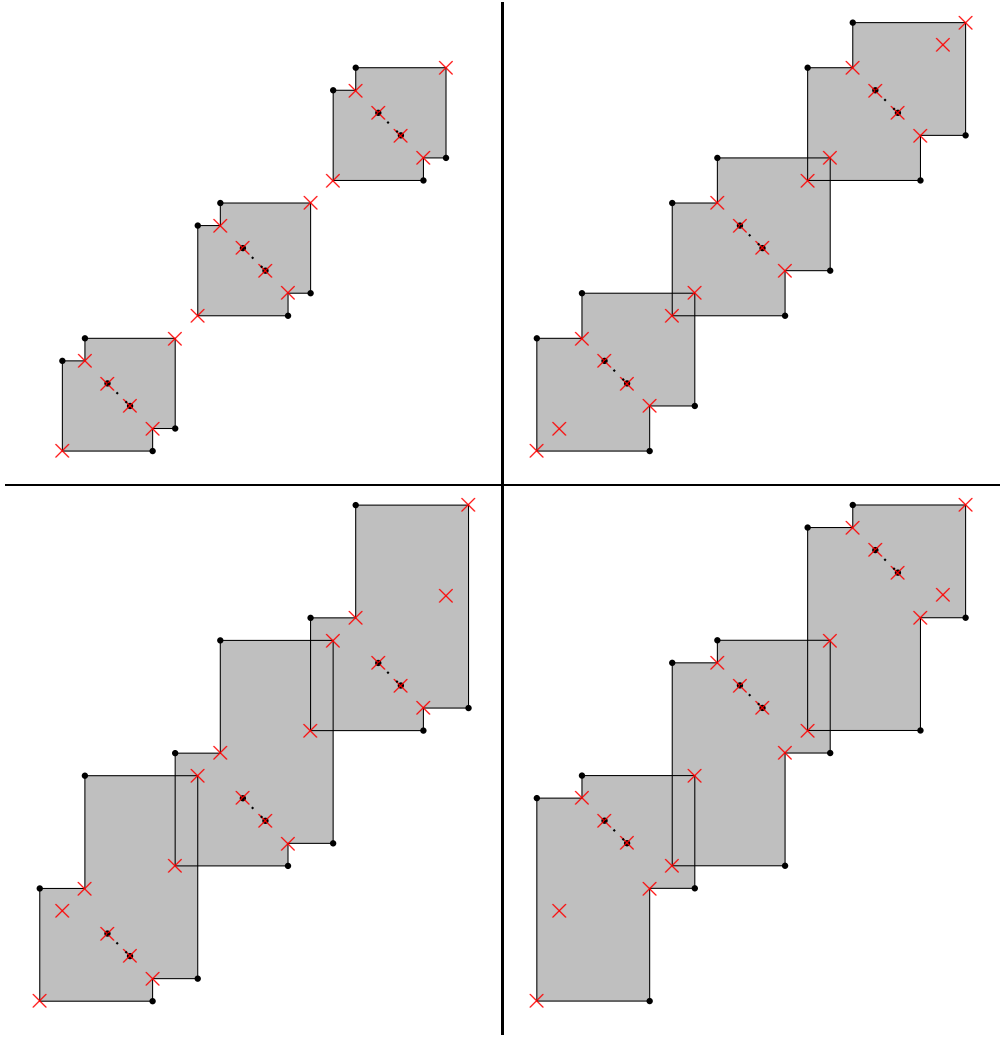

	It remains to prove the claim that these are all possible arrangements.
	If either $i_2<i_3+n$, or $\myomega_{i_4}<\myomega_{i_1+n}$,
	then the shaded regions do not overlap,
	as seen in the first picture of Figure~\ref{fig:alldistinct}.   In this case, 
	$\ell(\myomega)-\ell(\mychi)$ and $\#\scrR(\mychi,\myomega)$ do not depend on how the other values intersperse the pattern. 	

	Suppose $i_2>i_3+n$ and $\myomega_{i_4}>\myomega_{i_1+n}$, so that the translated shaded regions do overlap.
	We must have $i_2<i_1+n$ or else we obtain a dot in one of the forbidden regions in Figure~\ref{fig:normalize45312}.
	Furthermore, the indices $i_4,i_4+1,\dots,i_4+k,i_1$ are all consecutive,
	which forces $i_3+n,i_2,i_4+n$ to also be consecutive.
	Hence,  $[\myomega_{i_3},\myomega_{i_2-n},\myomega_{i_4},\myomega_{i_4+1},\dotsc,\myomega_{i_4+k},\myomega_{i_1}]$ is a complete window.

	Observe from Figure~\ref{fig:normalize45312}, that the values $\myomega_{i_4+1},\dotsc,\myomega_{i_4+k}$ are all consecutive.
	Furthermore, we must also have $\myomega_{i_2}<\myomega_{i_1+n}$,
	and $\myomega_{i_4}<\myomega_{i_3+n}$, or else
	we obtain a dot in one of the forbidden regions in Figure~\ref{fig:normalize45312}.
	If $\myomega_{i_4+1}>\myomega_{i_2+n}$, then we may replace the 3412 pattern with
	$\myomega_{i_4+1-n}\myomega_{i_4+1}\myomega_{i_1}\myomega_{i_2}$ in the normalization procedure.
	Thus, the consecutive values $\myomega_{i_4+1},\dotsc,\myomega_{i_4+k}$ must all lie either between
	$\myomega_{i_2}$ and $\myomega_{i_1+n}$ or between $\myomega_{i_1+n}$ and $\myomega_{i_2+n}$.

	Consider first, when $\myomega_{i_2}<\myomega_{i_4+k}<\cdots<\myomega_{i_4+1}<\myomega_{i_1+n}$.
	By our normalization procedure, there are two possibilities for $\myomega_{i_4}$;
	either $\myomega_{i_1+n}<\myomega_{i_4}<\myomega_{i_2+n}$ or $\myomega_{i_4+1+n}<\myomega_{i_4}<\myomega_{i_3+n}$.
	These cases appear in pictures 2 and 3 of Figure~\ref{fig:alldistinct}.
	In either case, the relative value of $\myomega_{i_3}$ is forced in the interval
	$\myomega_{i_4+1}<\myomega_{i_3}<\myomega_{i_1+n}$, by the normalization procedure.

	Finally, consider the case where $\myomega_{i_1+n}<\myomega_{i_4+k}<\cdots<\myomega_{i_4+1}<\myomega_{i_2+n}$.
	If $\myomega_{i_4}>\myomega_{i_4+1+n}$, then $\myomega_{i_4}\myomega_{i_4+1}\myomega_{i_4+1+n}\myomega_{i_1+n}$ forms a 4231 pattern.
	If $\myomega_{i_2+n}<\myomega_{i_4}<\myomega_{i_4+1+n}$, then we could replace the 3412 pattern with
	$\myomega_{i_4-n}\myomega_{i_4+1}\myomega_{i_1}\myomega_{i_2}$ in the normalization procedure.
	Hence,  we must have $\myomega_{i_4}<\myomega_{i_2+n}$.
	This case corresponds to the last picture in Figure~\ref{fig:alldistinct}.
	Again, the relative value of $\myomega_{i_3}$ is forced.

\item  \label{case:2}
Two residues the same.\\
	There are only two cases to consider;
	either $i_3\equiv i_4\pmod{n}$ or $i_1\equiv i_2\pmod{n}$.
	Note that both cannot happen simultaneously,
	otherwise $\myomega_{i_3}\myomega_{i_4+1-n}\myomega_{i_4+1}\myomega_{i_1}$ gives a 4231 pattern.
	If $i_3\equiv i_4\pmod{n}$, then let $\mychi=\myomega t_{i_4i_1}t_{i_1i_2}t_{i_2,i_4+n}$.
	If instead $i_1\equiv i_2\pmod{n}$, then let $\mychi=\myomega t_{i_4i_1}t_{i_3i_4}t_{i_1,i_3+n}$.
	The picture for the case where $i_3\equiv i_4$ is given in Figure~\ref{fig:2thesame}.
	The picture for the other case can be obtained by turning Figure~\ref{fig:2thesame} upside-down.
	In either case we have $\ell(\myomega)-\ell(\mychi)=2k+3$ and $\#\scrR(\mychi,\myomega)=2k+4$.

	Note that Figure~\ref{fig:2thesame} represents the only possible interspersing arrangement in this case.  In particular, 
         we must have $\myomega_{i_3}<\myomega_{i_1+n}$.
	Otherwise, $\myomega_{i_1+n}<\myomega_{i_4+k}$ by Figure~\ref{fig:normalize45312},
	and $\myomega_{i_4}=\myomega_{i_3+n}>\myomega_{i_4+1+n}$,
	and hence $\myomega_{i_4}\myomega_{i_4+1}\myomega_{i_4+1+n}\myomega_{i_1+n}$ gives a 4231 pattern.
	Similarly, if $i_1\equiv i_2\pmod{n}$, then we must have $\myomega_{i_4}<\myomega_{i_2+n}$,
	or else $\myomega_{i_4}\myomega_{i_4+1}\myomega_{i_4+1+n}\myomega_{i_2}$ gives a 4231 pattern.

	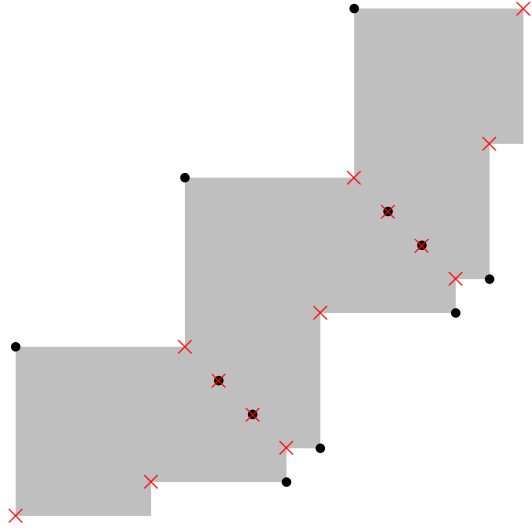
\begin{figure}[h]
	\begin{center}
	\begin{tikzpicture}
	\pgftransformscale{0.45}

	\coordinate (ineg4) at (-4,5);
	\coordinate (i1) at (1,10);
	\coordinate (i2) at (2,4);
	\coordinate (i3) at (3,3);
	\coordinate (i4) at (4,1);
	\coordinate (i5) at (5,2);
	\coordinate (i6) at (6,15);
	\coordinate (i7) at (7,9);
	\coordinate (i8) at (8,8);
	\coordinate (i9) at (9,6);
	\coordinate (i10) at (10,7);
	
	\coordinate (jneg4) at (-4,0);
	\coordinate (j0) at (0,1);
	\coordinate (j1) at (1,5);
	\coordinate (j2) at (2,4);
	\coordinate (j3) at (3,3);
	\coordinate (j4) at (4,2);
	\coordinate (j5) at (5,6);
	\coordinate (j6) at (6,10);
	\coordinate (j7) at (7,9);
	\coordinate (j8) at (8,8);
	\coordinate (j9) at (9,7);
	\coordinate (j10) at (10,11);
	\coordinate (j11) at (11,15);

	\fill [lightgray] (jneg4) -- (ineg4) -- (j1) -- (i1) -- (j6) -- (i6) -- (j11) -- (11,11) -- (j10) -- (i10) -- (j9) -- (i9) --
		(j5) -- (i5) -- (j4) -- (i4) -- (j0) -- (0,0) -- (jneg4);

	\foreach \point in {ineg4,i1,i2,i3,i4,i5,i6,i7,i8,i9,i10}
		{\fill [black] (\point) circle (4pt);}

	\foreach \point in {jneg4,j0,j1,j2,j3,j4,j5,j6,j7,j8,j9,j10,j11}
		{\draw [red] (\point) node {$\times$};}
	\end{tikzpicture}
	\end{center}
	\caption{Case with $i_3\equiv i_4\pmod{n}$.}
	\label{fig:2thesame}
	\end{figure}
\end{enumerate}
	
\subsection{When Region $B$ is Empty.}
By the normalization in Figure~\ref{fig:normalize}, we may assume $i_1=i_4+1$ since region $B$ is empty.
\begin{enumerate}[\text{Case }1:]
\item Either $i_3\not\equiv i_4$ or $i_1\not\equiv i_2\pmod{n}$.\\
	In this case we do not flatten the affine permutation and use Bruhat pictures.
	Instead, we use the methods and notation from Section~\ref{Converse} to show directly that
	the Poincar\'e polynomial fails to be palindromic at degree 1.
	Fix $\beta$ such that $\myomega_\beta=\max_{j\le i_3}\{\myomega_j\}$.
	Hence,  $\beta\le i_3<\beta+n$.
	Construct the graph $G_\beta$, as in Section~\ref{Converse},
	and let $e_1,\dotsc,e_r$ be the extra edges from Lemma~\ref{notproper}.
	By Proposition~\ref{propercase}, we may assume $\myomega$ is not contained in a proper parabolic subgroup.
	Hence, we need only show that $|E_\beta|>n$.
	By Lemma~\ref{notproper}, $|E_{\beta}| \geq n$ so we need only show
	there exists one edge not in $H_{\beta}$ or in $\{e_1,\ldots,e_r\}$. 
	Suppose $i_3$ is a vertex in $H_{i,\beta}\subset G_\beta$.
	There are three ways the 3412 pattern can occur in $G_\beta$ in this case.
	\begin{enumerate}[(a)]
	\item \label{case1} $i_4<i_3+n$.\\
		If $i_2-i_3<n$, then $\gamma=(i_3,i_2)$ is an edge in $G_\beta$, by Proposition~\ref{BruhatOrder}.
		Otherwise, there exists some $m\ge0$ such that $0<\myomega_{i_3}-\myomega_{i_2+mn}<n$,
		in which case $\gamma=(i_3,i_2+mn)$ is an edge in $G_\beta$.
		Since $i_1=i_4+1$ and $i_4<i_3+n$, we also have $i_1<i_3+n$.
		Hence, there is a sequence $i_3=j_1<j_2<\cdots<j_p=i_1$, such that
		$\ell(\myomega t_{j_m,j_{m+1}})=\ell(\myomega)-1$ for all $1\le m<p$, by Lemma~\ref{chain} which gives 
                us a path in $G_{\beta}$ from $i_{3}$ to $i_{1}$.  
		Let $\delta$ be the first edge in this path not contained entirely in $H_{i,\beta}$.
		Since $i_3<i_4<i_1$ and $\myomega_{i_4}>\myomega_{i_3}$,
		$i_3$ and $i_1$ cannot both be in $H_{i,\beta}$,
		we conclude such an edge $\delta$ exists.
		At most one of $\gamma$ and $\delta$ can be the added edge $e_i$
		and neither edge is in $H_\beta$.

	\item \label{case2} $i_4\ge i_3+n$ and $i_1\not\equiv i_2\pmod{n}$.\\
		There exist $u,v\ge0$ such that $0<\myomega_{i_3}-\myomega_{i_1+un}<n$ and $0<\myomega_{i_3}-\myomega_{i_2+vn}<n$.
		By the way we have normalized the 3412 pattern, there is no $i_3<j<i_2+vn$ with $\myomega_{i_2+vn}<\myomega_j<\myomega_{i_3}$,
		since $\myomega_{j-vn}$ cannot be in the dotted region of Figure~\ref{fig:normalize}.
		Hence, $\gamma=(i_3,i_2+vn)$ is an edge in $G_\beta$.
		
		\hspace*{6pt} 
	  	To identify a second edge in $G_{\beta}$, 
                we must consider the different ways $\pt_\myomega(i_1+un)$ and $\pt_\myomega(i_2+vn)$ are arranged.
		If $\pt_\myomega(i_1+un)$ is southwest or northeast of $\pt_\myomega(i_2+vn)$,
		then let $\delta$ be the first edge in the chain from $i_3$ to $i_1+un$ found in Lemma~\ref{chain}, that is not contained in $H_{i,\beta}$.
		Then $\gamma\not=\delta$, since their right endpoints are guaranteed to be different.
		At most one of $\gamma$ and $\delta$ can be the added edge $e_i$
		and neither edge is in $H_\beta$.
		
		\hspace*{6pt} Suppose instead, $\pt_\myomega(i_1+un)$ is southeast of $\pt_\myomega(i_2+vn)$.
		Then $i_1+un>i_2+vn$ and $\myomega_{i_1+un}<\myomega_{i_2+vn}$ and hence
		$i_1>i_2+(v-u)n$ and $\myomega_{i_1}<\myomega_{i_2+(v-u)n}$.
		Recall we have normalized our permutation so that either $i_2-i_1<n$ or $\myomega_{i_2}-\myomega_{i_1}<n$.
		But the southeastern configuration can only happen if $i_2-i_1<n$.
		Hence,  $i_3<i_2-n<i_4=i_3+n$.
		
		\hspace*{6pt} If $i_2-n$ is not a vertex in $H_{i,\beta}$, then let $\delta$ be the first edge in the
		sequence from $i_3$ to $i_2-n$ that leaves $H_{i,\beta}$.
		Otherwise, if $i_2-n$ is a vertex in $H_{i,\beta}$, then let $\delta$ be the first edge
		in the sequence from $i_2-n$ to $i_1$ that leaves $H_{i,\beta}$.
		In either case, $\delta$ is different from $\gamma$, since its right endpoint is smaller.
		At most one of $\gamma$ and $\delta$ can be the added edge $e_i$
		and neither edge is in $H_\beta$.

	\item \label{case3} $i_4>i_3+n$ and $i_1\equiv i_2\pmod{n}$.\\
		As in Case~\ref{case2}, there exists some $s\ge1$ such that $0<\myomega_{i_3}-\myomega_{i_1+sn}<n$
		and the edge $\gamma=(i_3,i_1+sn)$ is in $G_\beta$.
		Since $\myomega_{i_4}>\myomega_{i_2}=\myomega_{i_1+n}$, we have $\myomega_{i_4-n}>\myomega_{i_1}$.
		Since $i_4-i_3>n$, we must have $\myomega_{i_4}-\myomega_{i_3}<n$.
		In fact, by the normalization of the 3412 pattern, this forces $\myomega_{i_1}<\myomega_{i_4-n}<\myomega_{i_2}$.
		There are two cases to consider.
		
		\hspace*{6pt} If $i_4-n$ is not a vertex in $H_{i,\beta}$,
		then let $\delta$ be the first edge in the sequence from $i_3$ to $i_4-n$ that leaves $H_{i,\beta}$.
		If instead, $i_4-n$ is a vertex in $H_{i,\beta}$,
		then let $\delta$ be the first edge in the sequence from $i_4-n$ to $i_1$ that leaves $H_{i,\beta}$.
		In either case, $\delta$ if different from $\gamma$, since the right end point of $\delta$ is smaller than $\myomega_{i_2}$.
		At most one of $\gamma$ and $\delta$ can be the added edge $e_i$
		and neither edge is in $H_\beta$.
	\end{enumerate}

\item Every normalized 3412 pattern has $i_3\equiv i_4$ and $i_1\equiv i_2\pmod{n}$.\\
	We can assume that $i_{3}=i_{4}-n$ and $i_{2}=i_{1}+n$.  
	Then $\myomega$ is as in Figure~\ref{fig:3412nodistinct} where region $E$ is the rectangle with corners at 
         $(i_{1}-n, \myomega_{i_{4}-2n})$ and $(i_{4}, \myomega_{i_{1}+n})$.     
	Since $\myomega$ avoids 4231, either all of the values in region $E$ are decreasing or $E$ is empty.

	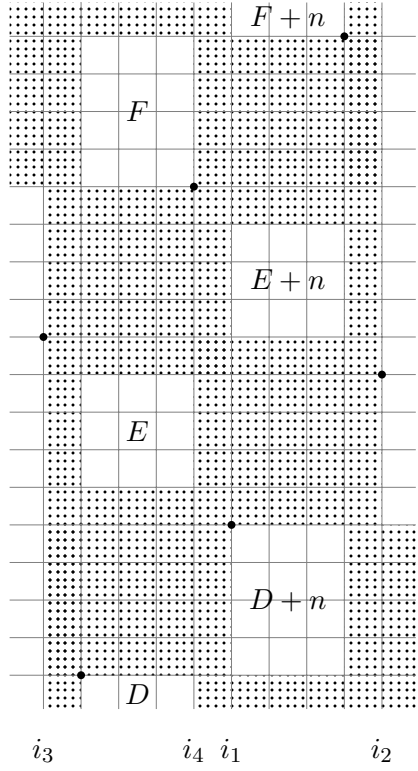
\begin{figure}[h]
	\begin{center}
	\begin{tikzpicture}
	\pgftransformscale{0.5}

	\coordinate (i3) at (0,9);
	\coordinate (i4) at (4,13);
	\coordinate (i5) at (8,17);
	\coordinate (i0) at (1,0);
	\coordinate (i1) at (5,4);
	\coordinate (i2) at (9,8);

	\fill [pattern=dots] (0,-0.9) -- (0,13) -- (-0.9,13) -- (-0.9,17.9) -- (5,17.9) -- (5,17) -- (1,17) -- (1,-0.9) -- (0,-0.9);
	\fill [pattern=dots] (1,8) -- (1,13) -- (5,13) -- (5,8) -- (1,8);
	\fill [pattern=dots] (0,0) -- (0,4) -- (5,4) -- (5,0) -- (0,0);
	\fill [pattern=dots] (4,0) -- (4,-0.9) -- (9.9,-0.9) -- (9.9,4) -- (8,4) -- (8,0) -- (4,0);
	\fill [pattern=dots] (4,4) -- (4,9) -- (9,9) -- (9,4) -- (4,4);
	\fill [pattern=dots] (4,13) -- (4,17) -- (9,17) -- (9,13) -- (4,13);
	\fill [pattern=dots] (8,9) -- (8,17.9) -- (9,17.9) -- (9,9) -- (8,9);
	\fill [pattern=dots] (1,4) -- (1,5) -- (4,5) -- (4,4) -- (1,4);
	\fill [pattern=dots] (5,12) -- (5,13) -- (8,13) -- (8,12) -- (4,12);
	\draw[step=1,gray,very thin] (-0.9,-0.9) grid (9.9,17.9);
	\foreach \point in {i3,i4,i5,i0,i1,i2}
		{\fill [black] (\point) circle (3pt);}

	\draw (0,-2) node{$i_3$};
	\draw (4,-2) node{$i_4$};
	\draw (5,-2) node{$i_1$};
	\draw (9,-2) node{$i_2$};
	\draw (2.5,-0.5) node{$D$};
	\draw (2.5,6.5) node{$E$};
	\draw (2.5,15) node{$F$};
	\draw (6.5,2) node{$D+n$};
	\draw (6.5,10.5) node{$E+n$};
	\draw (6.5,17.5) node{$F+n$};

	\end{tikzpicture}
	\end{center}
	\caption{A normalized $3412$ pattern with only two distinct residues.}
	\label{fig:3412nodistinct}
	\end{figure}

	\begin{enumerate}[(a)]
	\item Region $E$ is nonempty.\\
		There exists some $i_1-n<j<i_4$ with $\myomega_{i_1}<\myomega_j<\myomega_{i_1+n}$.
		By the normalization process, $\myomega_{i_4-n}<\myomega_{j+n}<\myomega_{i_1+2n}$,
		since there is no $i_1<k<i_1+n$ with $\myomega_{i_1+n}<\myomega_k<\myomega_{i_4-n}$.
		Hence, $\myomega_{i_4+rn}<\myomega_{i_1+(r+1)n}<\myomega_{i_4+(r-1)n}$ for all $r\in\Z$.
		Let $\mychi=\myomega t_{i_4-n,i_1}t_{i_4,i_1}$.
		Replace $\mychi,\myomega$, with the flattened pair $\wt{\mychi},\wt{\myomega}$.
		Any indices with values in regions $D$ or $F$, aside from $i_4$ and $i_1$,
		are removed during the flattening process,
		and we are left with the picture in Figure~\ref{fig:boththesameE}.
		Suppose there are $k\ge1$ (decreasing) indices in region $E$.
		Then $\ell(\myomega)-\ell(\mychi)=2k+2$ and $\#\scrR(\mychi,\myomega)=4k+2$.

		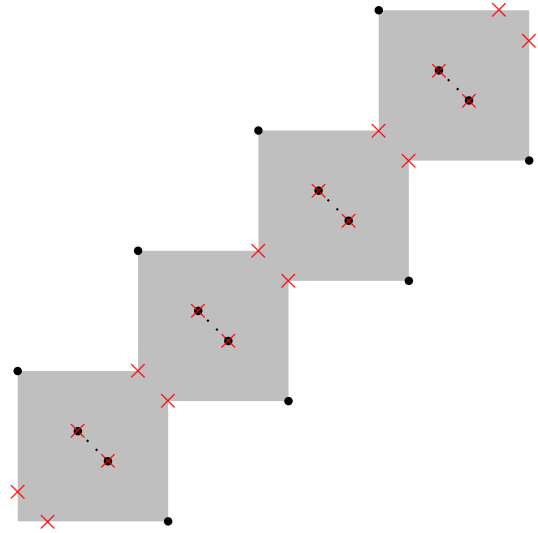
\begin{figure}
		\begin{center}
		\begin{tikzpicture}
		\pgftransformscale{0.4}

		\coordinate (ineg3) at (-3,2);
		\coordinate (ineg1) at (-1,0);
		\coordinate (i0) at (0,-1);
		\coordinate (i1) at (1,6);
		\coordinate (i2) at (2,-3);
		\coordinate (i3) at (3,4);
		\coordinate (i4) at (4,3);
		\coordinate (i5) at (5,10);
		\coordinate (i6) at (6,1);
		\coordinate (i7) at (7,8);
		\coordinate (i8) at (8,7);
		\coordinate (i9) at (9,14);
		\coordinate (i10) at (10,5);
		\coordinate (i11) at (11,12);
		\coordinate (i12) at (12,11);
		\coordinate (i14) at (14,9);

		\coordinate (jneg3) at (-3,-2);
		\coordinate (jneg2) at (-2,-3);
		\coordinate (jneg1) at (-1,0);
		\coordinate (j0) at (0,-1);
		\coordinate (j1) at (1,2);
		\coordinate (j2) at (2,1);
		\coordinate (j3) at (3,4);
		\coordinate (j4) at (4,3);
		\coordinate (j5) at (5,6);
		\coordinate (j6) at (6,5);
		\coordinate (j7) at (7,8);
		\coordinate (j8) at (8,7);
		\coordinate (j9) at (9,10);
		\coordinate (j10) at (10,9);
		\coordinate (j11) at (11,12);
		\coordinate (j12) at (12,11);
		\coordinate (j13) at (13,14);
		\coordinate (j14) at (14,13);

		\fill [lightgray] (-3,-3) -- (ineg3) -- (j1) -- (i1) -- (j5) -- (i5) -- (j9) -- (i9) -- (14,14) -- (i14) --
			(j10) -- (i10) -- (j6) -- (i6) -- (j2) -- (i2) -- (-3,-3);
		\foreach \point in {ineg3,ineg1,i0,i1,i2,i3,i4,i5,i6,i7,i8,i9,i10,i11,i12,i14}
			{\fill [black] (\point) circle (4pt);}
		\foreach \point in {jneg3,jneg2,jneg1,j0,j1,j2,j3,j4,j5,j6,j7,j8,j9,j10,j11,j12,j13,j14}
			{\draw [red] (\point) node {$\times$};}
		\draw[loosely dotted,thick] (i3) -- (i4);
		\draw[loosely dotted,thick] (i7) -- (i8);
		\draw[loosely dotted,thick] (i11) -- (i12);
		\draw[loosely dotted,thick] (ineg1) -- (i0);

		\end{tikzpicture}
		\end{center}
		\caption{Region $E$ is nonempty.}
		\label{fig:boththesameE}
		\end{figure}

	\item Region $E$ is empty, but regions $D$ and $F$ are both nonempty.\\
	        We can assume that $\pt_w(i_{4}-1)$ is in $F$ and $\pt_{w}(i_{1}+1)$ is in region $D+n$ since every normalized
		3412 pattern has only two congruence classes among ${i_{1},i_{2},i_{3},i_{4}}$ and $E$ is empty.    
		Let $\mychi=\myomega t_{i_4-n-1,i_1+1}t_{i_4-1,t_1+1}t_{i_4-n,i_1}t_{i_4i_1}$.
		Replace $\mychi,\myomega$, with the flattened pair $\wt{\mychi},\wt{\myomega}$,
		so that we have flattened onto the consecutive indices $\{i_4-1,i_4,i_1,i_1+1\}$.
		Fix $k\geq 3$,  such that $0<\myomega_{i_4}-\myomega_{i_1+kn}<n$.
		Figure~\ref{fig:boththesameDF} shows the picture when $k=3$ and $k=4$.
		Regardless of $k$, we have $\ell(\myomega)-\ell(\mychi)=8$ and $\#\scrR(\mychi,\myomega)=12$.

		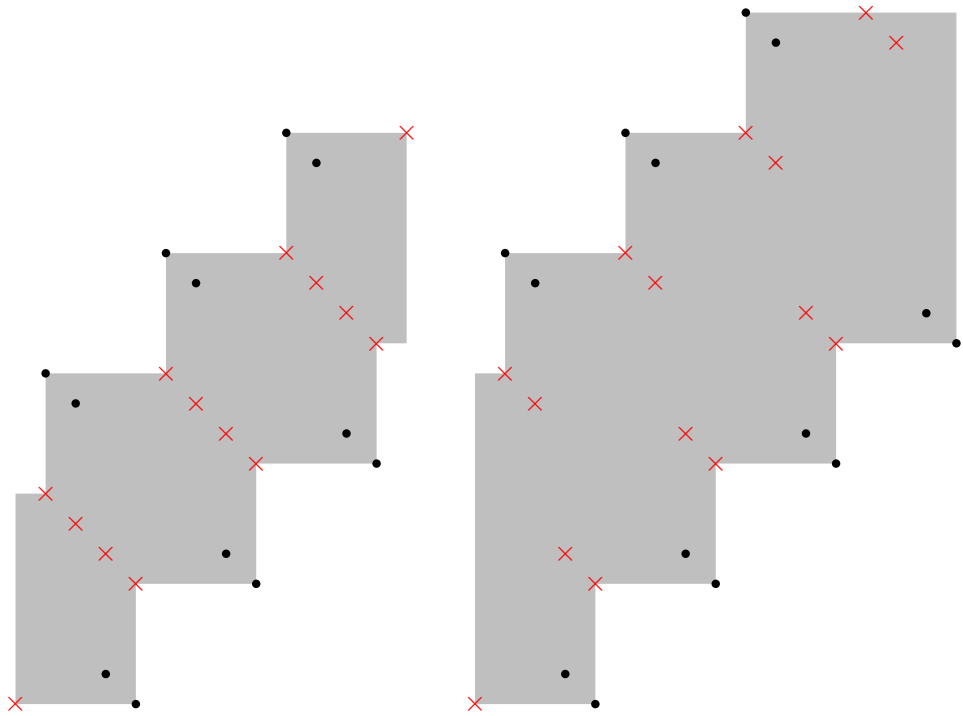
\begin{figure}
		\begin{displaymath}
		\begin{array}{cc}
		\begin{tikzpicture}
		\pgftransformscale{0.4}

		\coordinate (ineg3) at (-3,4);
		\coordinate (ineg2) at (-2,3);
		\coordinate (ineg1) at (-1,-6);
		\coordinate (i0) at (0,-7);
		\coordinate (i1) at (1,8);
		\coordinate (i2) at (2,7);
		\coordinate (i3) at (3,-2);
		\coordinate (i4) at (4,-3);
		\coordinate (i5) at (5,12);
		\coordinate (i6) at (6,11);
		\coordinate (i7) at (7,2);
		\coordinate (i8) at (8,1);

		\coordinate (jneg4) at (-4,-7);
		\coordinate (jneg3) at (-3,0);
		\coordinate (jneg2) at (-2,-1);
		\coordinate (jneg1) at (-1,-2);
		\coordinate (j0) at (0,-3);
		\coordinate (j1) at (1,4);
		\coordinate (j2) at (2,3);
		\coordinate (j3) at (3,2);
		\coordinate (j4) at (4,1);
		\coordinate (j5) at (5,8);
		\coordinate (j6) at (6,7);
		\coordinate (j7) at (7,6);
		\coordinate (j8) at (8,5);
		\coordinate (j9) at (9,12);
		
		\fill [lightgray] (jneg4) -- (-4,0) -- (jneg3) -- (ineg3) -- (j1) -- (i1) -- (j5) -- (i5) -- (j9) --
			(9,5) -- (j8) -- (i8) -- (j4) -- (i4) -- (j0) -- (i0) -- (jneg4);
		\foreach \point in {ineg3,ineg2,ineg1,i0,i1,i2,i3,i4,i5,i6,i7,i8}
			{\fill [black] (\point) circle (4pt);}
		\foreach \point in {jneg4,jneg3,jneg2,jneg1,j0,j1,j2,j3,j4,j5,j6,j7,j8,j9}
			{\draw [red] (\point) node {$\times$};}
		\end{tikzpicture}
		&
		\begin{tikzpicture}
		\pgftransformscale{0.4}

		\coordinate (ineg3) at (-3,6);
		\coordinate (ineg2) at (-2,5);
		\coordinate (ineg1) at (-1,-8);
		\coordinate (i0) at (0,-9);
		\coordinate (i1) at (1,10);
		\coordinate (i2) at (2,9);
		\coordinate (i3) at (3,-4);
		\coordinate (i4) at (4,-5);
		\coordinate (i5) at (5,14);
		\coordinate (i6) at (6,13);
		\coordinate (i7) at (7,0);
		\coordinate (i8) at (8,-1);
		\coordinate (i11) at (11,4);
		\coordinate (i12) at (12,3);

		\coordinate (jneg4) at (-4,-9);
		\coordinate (jneg3) at (-3,2);
		\coordinate (jneg2) at (-2,1);
		\coordinate (jneg1) at (-1,-4);
		\coordinate (j0) at (0,-5);
		\coordinate (j1) at (1,6);
		\coordinate (j2) at (2,5);
		\coordinate (j3) at (3,0);
		\coordinate (j4) at (4,-1);
		\coordinate (j5) at (5,10);
		\coordinate (j6) at (6,9);
		\coordinate (j7) at (7,4);
		\coordinate (j8) at (8,3);
		\coordinate (j9) at (9,14);
		\coordinate (j10) at (10,13);
		
		\fill [lightgray] (jneg4) -- (-4,2) -- (jneg3) -- (ineg3) -- (j1) -- (i1) -- (j5) -- (i5) -- (12,14) --
			(i12) -- (j8) -- (i8) -- (j4) -- (i4) -- (j0) -- (i0) -- (jneg4);
		\foreach \point in {ineg3,ineg2,ineg1,i0,i1,i2,i3,i4,i5,i6,i7,i8,i11,i12}
			{\fill [black] (\point) circle (4pt);}
		\foreach \point in {jneg4,jneg3,jneg2,jneg1,j0,j1,j2,j3,j4,j5,j6,j7,j8,j9,j10}
			{\draw [red] (\point) node {$\times$};}
		\end{tikzpicture}
		\end{array}
		\end{displaymath}
		\caption{Region $E$ is empty, but regions $D$ and $F$ are nonempty ($k=3,4$).}
		\label{fig:boththesameDF}
		\end{figure}

	\item Region $E$ is empty, and exactly one of regions $D$ and $F$ is empty.\\
		Note, one of them must be nonempty whenever $n\ge3$.
          	Assume, $D$ is nonempty, the other case being similar.  

		\hspace*{6pt}If there is an ascent among the values in region $D$, choose a
		leftmost pair $i_{1}<a,b<i_{4}+n$ such that $w_{a}<w_{b}$.
		For example, one could let $a>i_{1}$ be the smallest index such that there
		exists a $a<j<i_{4}+n$ with $w_{a}<w_{j}<w_{i_{1}}$, and let $b>a$ be
		the smallest index such that $w_{a}<w_{b}$.
		Then $w_{i_{4}-n}w_{i_{4}}w_{i_{1}}w_{a}w_{b}$ forms a $45312_{k}$ pattern.
		Following the same reasoning as in Case 2 on Page~\pageref{case:2},
		let $\mychi=\myomega t_{i_4,a}t_{a,b}t_{b,i_4+n}$ and replace
		$\mychi,\myomega$ by the corresponding flattened pair.
		The affine Bruhat picture for $k=2$ is shown in Figure~\ref{fig:DhasAscent}.

		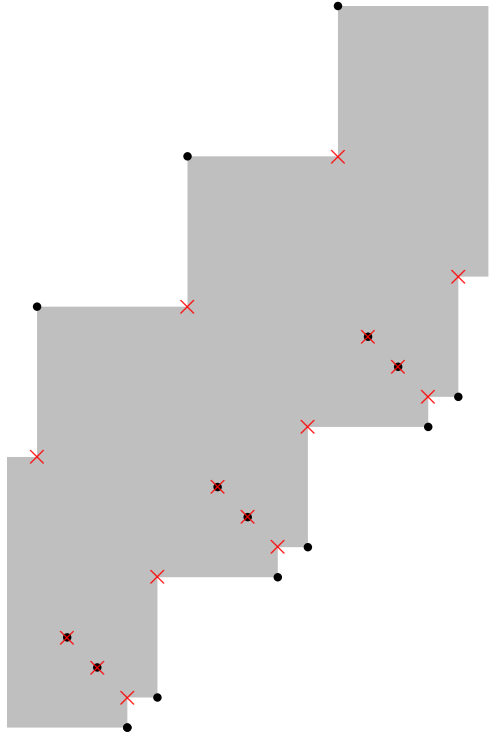
\begin{figure}
		\begin{center}
		\begin{tikzpicture}
		\pgftransformscale{0.4}

		\coordinate (ineg4) at (-4,8);
		\coordinate (ineg3) at (-3,-3);				
		\coordinate (ineg2) at (-2,-4);
		\coordinate (ineg1) at (-1,-6);
		\coordinate (i0) at (0,-5);
		\coordinate (i1) at (1,13);
		\coordinate (i2) at (2,2);
		\coordinate (i3) at (3,1);
		\coordinate (i4) at (4,-1);
		\coordinate (i5) at (5,0);
		\coordinate (i6) at (6,18);
		\coordinate (i7) at (7,7);
		\coordinate (i8) at (8,6);
		\coordinate (i9) at (9,4);
		\coordinate (i10) at (10,5);

		\coordinate (jneg5) at (-5,-6);
		\coordinate (jneg4) at (-4,3);
		\coordinate (jneg3) at (-3,-3);
		\coordinate (jneg2) at (-2,-4);
		\coordinate (jneg1) at (-1,-5);
		\coordinate (j0) at (0,-1);
		\coordinate (j1) at (1,8);
		\coordinate (j2) at (2,2);
		\coordinate (j3) at (3,1);
		\coordinate (j4) at (4,0);
		\coordinate (j5) at (5,4);
		\coordinate (j6) at (6,13);
		\coordinate (j7) at (7,7);
		\coordinate (j8) at (8,6);
		\coordinate (j9) at (9,5);
		\coordinate (j10) at (10,9);
		\coordinate (j11) at (11,18);

		\fill [lightgray] (jneg4) -- (ineg4) -- (j1) -- (i1) -- (j6) -- (i6) -- (j11) -- (11,9) -- (j10) -- (i10) --
			(j9) -- (i9) -- (j5) -- (i5) -- (j4) -- (i4) -- (j0) -- (i0) -- (jneg1) -- (ineg1) -- (jneg5) -- (-5,3) -- (jneg4);
		\foreach \point in {ineg4,ineg3,ineg2,ineg1,ineg1,i0,i1,i2,i3,i4,i5,i6,i7,i8,i9,i10}
			{\fill [black] (\point) circle (4pt);}
		\foreach \point in {jneg4,jneg3,jneg2,jneg1,j0,j1,j2,j3,j4,j5,j6,j7,j8,j9,j10}
			{\draw [red] (\point) node {$\times$};}

		\end{tikzpicture}
		\end{center}
		\caption{Region $E$ is empty and region $D$ has an ascent (k=2).}
		\label{fig:DhasAscent}
		\end{figure}

		\hspace*{6pt}For all $k\geq 1$, we have $\ell(\myomega)-\ell(\mychi)=2k+3$ and $\#\scrR(\mychi,\myomega)=2k+4$.
		The picture for the case $F$ is nonempty and contains an ascent can be
		obtained by turning Figure~\ref{fig:DhasAscent} upside-down.

		\hspace*{6pt}Finally, if the values in region $D$ are all decreasing, then the
		entries must be $n-1$ consecutive values in consecutive positions,
		or else either region $D$ or $E$ will be nonempty since $w$ is a bijection.
		In this case, we claim $\myomega$ is a twisted spiral permutation.

		\hspace*{6pt}To prove the claim, let $k\geq 2$ be the integer such that $\myomega_{i_1+kn}=\myomega_{i_4}-1$.
		Then $\myomega \cdot c(i_{4},k(n-1))=w_{0}$, using the notation from Section~\ref{s:spirals}.  
		Hence, by \eqref{e:spiralswap}, $\myomega$ is a twisted spiral permutation.
	\end{enumerate}
\end{enumerate}
\end{proof}

\section{Further Directions}
\label{conjectures}

Unlike rationally smooth Schubert varieties, there are only finitely many smooth Schubert varieties corresponding to affine permutations.
In the proof of Theorem~\ref{thm:onedirection}, we saw that any 3412 and 4231 avoiding affine permutation can be written in the form
$\myomega=\myomega^\prime\mysigma$, with $\ell(\myomega)=\ell(\myomega^\prime)+\ell(\mysigma)$.
Both $\myomega^\prime$ and $\mysigma$ are elements of a proper parabolic subgroup of $\wt{S}_n$,
and hence $\ell(\myomega^\prime),\ell(\mysigma)\le\binom{n}{2}$.
Thus, we have the following corollary.

\begin{cor}[To Theorem~\ref{thm:onedirection}]
If $\myomega\in\wt{S}_n$ avoids 3412 and 4231, then $\ell(\myomega)\le2\binom{n}{2}$.
\end{cor}


Since the number of affine permutations of length at most $2\binom{n}{2}$ is finite,
we would like to compute how many affine permutations in $\wt{S}_n$ avoid both 3412 and 4231.
Conjecturally, this is equivalent to the number of smooth affine Schubert varieties of type $\wt{A}_n$.
Starting with $n=2$, the first few terms of this sequence are $5,31,173,891,4373$, which did not previously appear in Sloane's \cite{Sloanes}.

\smallskip

In \cite{BE:05}, Bj\"orner and Ekedahl give general inequalities
amongst the coefficients of the Poincar\'{e} polynomial for elements
of any crystallographic Coxeter group.  In Theorem~\ref{thm:4231}, we
prove $n=c_{1}<c_{\ell(w)-1}$.
Combining this fact with \cite[Theorem C]{BE:05}, proves the following corollary.

\begin{cor}
Let $\myomega \in \wt{S}_{n} $ and assume $w$ contains a 4231.  Then
if $P_{id,w}=1+a_{1}q+ \dotsb + a_{d}q^{d}$ is the Kazhdan-Lusztig
polynomial indexed by $id,w$, then $a_{1}>0$.
\end{cor}

\bibliographystyle{siam}

\end{document}